\documentclass[11pt]{amsart}
\usepackage{amscd,graphicx,psfrag,amsxtra,amsmath}

\newcommand{\p}{\partial}
\newcommand{\Z}{\mathbb Z}
\newcommand{\E}{\mathcal E}
\newcommand{\Q}{\mathbb Q}
\renewcommand{\H}{\mathcal H}
\newcommand{\G}{\mathcal G}

\newcommand{\D}{\mathcal D}
\newcommand{\Sm}{\mathcal S}

\newcommand{\B}{\,\mathcal B}
\newcommand{\C}{\mathbb C}
\newcommand{\M}{\mathcal M}
\newcommand{\R}{\mathbb R}

\newcommand{\im}{\operatorname{im}}

\newcommand{\ep}{\varepsilon}
\renewcommand{\phi}{\varphi}

\newcommand{\SW}{\operatorname{SW}\,}
\newcommand{\FO}{\operatorname{FO}\,}

\newcommand{\ind}{\operatorname{ind}}
\newcommand{\coker}{\operatorname{coker}}
\renewcommand{\Re}{\operatorname{Re}}
\renewcommand{\Im}{\operatorname{Im}}
\newcommand{\rk}{\operatorname{rank}}

\newcommand{\Res}{\operatorname{Res}}

\newcommand{\Map}{\operatorname{Map}\,}

\newcommand{\sign}{\operatorname{sign}}

\renewcommand{\P}{\mathcal{P}}

\newcommand{\btilde}{\beta}
\newcommand{\mZ}{\mathcal Z}
\newcommand{\tmZ}{\widetilde{\mathcal Z}}
\newcommand{\SF}{\operatorname{SF}}
\newcommand{\met}{\operatorname{{\mathcal R}}}
\newcommand{\swz}{\chi} 
\newcommand{\Dir}{\,\operatorname{Dir}}
\newcommand{\Sign}{\,\operatorname{Sign}}
\newcommand{\Spec}{\,\operatorname{Spec}}
\newcommand{\Ss}{\it\Sigma}
\newcommand{\spinc}{\ifmmode{\operatorname{Spin}^c}\else{$\operatorname{spin}^c$\ }\fi}
\newcommand{\spincs}{\mathfrak s}

\newtheorem{theorem}{Theorem}[section]
\newtheorem{thm}{Theorem}

\newtheorem{theorem2}[thm]{Conjecture}
\newtheorem{lemma}[theorem]{Lemma}
\newtheorem{proposition}[theorem]{Proposition}
\newtheorem{corollary}[theorem]{Corollary}
\theoremstyle{definition}

\newtheorem{remark}[theorem]{Remark}

\title{Seiberg--Witten equations, end-periodic Dirac 
operators, and a lift of Rohlin's invariant}

\thanks{The first author was partially supported by NSF Grant
0805841, the second author was partially supported by NSF
Grant 0804760, and the third author was partially supported by NSF
Grant 0305946 and the Max-Planck-Institut f\"ur Mathematik in Bonn,
Germany}
\author[Tomasz Mrowka]{Tomasz Mrowka}
\address{Department of Mathematics\newline\indent Massachusetts Institute of 
Technology, Cambridge MA 02139}
\email{\rm{mrowka@mit.edu}}
\author[Daniel Ruberman]{Daniel Ruberman}
\address{Department of Mathematics, MS 050\newline\indent Brandeis
University \newline\indent Waltham, MA 02454}
\email{\rm{ruberman@brandeis.edu}}
\author[Nikolai Saveliev]{Nikolai Saveliev}
\address{Department of Mathematics\newline\indent
University of Miami \newline\indent PO Box 249085
\newline\indent Coral Gables, FL 33124}
\email{\rm{saveliev@math.miami.edu}}
\begin{document}
\begin{abstract}
We introduce a gauge-theoretic integer valued lift of the Rohlin invariant of a smooth $4$-manifold $X$ with the homology of $S^1 \times S^3$.  The invariant has two terms; one is a count of solutions to the Seiberg-Witten equations on $X$, and the other is essentially the index of the Dirac operator on a non-compact manifold with end modeled on the infinite cyclic cover of $X$.  Each term is metric (and perturbation) dependent, and we show that these dependencies cancel as the metric and perturbation vary in a generic 1-parameter family.
\end{abstract}
\maketitle

\section{Introduction}\label{S:intro}
We use the Seiberg-Witten equations to define an integer valued invariant of 
an oriented smooth closed $4$-manifold $X$ with the integral homology of $S^1\times 
S^3$ that reduces mod $2$ to the Rohlin invariant~\cite{ruberman:ds,ruberman-saveliev:survey,scharlemann:phs} of $X$. Manifolds with this homology type arise, for instance, from furling up a homology cobordism from an integral homology sphere to itself, hence their study should shed light on classical problems concerning the homology cobordism group $\Theta_3^H$ of integral homology spheres and the classical Rohlin invariant $\rho: \Theta_3^H \to \mathbb Z/2$. In addition, the basic example of such a manifold, namely $S^1\times S^3$ itself, is perhaps the simplest smooth $4$-manifold that might admit an exotic smooth structure, detectable in principle by the Rohlin invariant (the Kirby--Siebenmann invariant~\cite{kirby-siebenmann} in this dimension). 

Because $b^2_+(X) = 0$, the usual count of solutions to the Seiberg-Witten equations on $X$ will depend on choices of metric and perturbation. The main work in the paper is to define and understand a correction term to this count, based on a study of the spin Dirac operator on a non-compact manifold with end modeled on the infinite cyclic cover of $X$.   The index of this operator is defined using Taubes' theory~\cite{T} of periodic-end operators; we extend this theory by proving a change of index formula for 1-parameter families of such operators in terms of version of spectral flow.

In summary, the invariant is defined as follows. The Seiberg--Witten equations 
on $X$ depend on two parameters: a Riemannian metric $g$ and a perturbing 
self-dual $2$-form, written as $d^+\beta$. For generic choices of these 
parameters, the Seiberg--Witten moduli space $\M(X,g,\beta)$ consists only of 
irreducible points, which can be counted with signs depending on a choice of 
orientation and homology orientation on $X$. This count, denoted $\# \M (X,g,
\beta)$, will generally depend on the parameters $(g,\beta)$ because reducible 
solutions can appear along a path connecting two sets of generic parameters. To 
remedy this, we introduce an index-theoretic correction term $w\,(X,g,\beta)$ 
whose definition depends on an understanding of the Dirac operator on a 
non-compact manifold with end modeled on the infinite cyclic cover of $X$, 
using Taubes' theory~\cite{T} of end-periodic operators and some additional 
analysis. By proving that the change in $w\,(X,g,\beta)$ (the aforementioned 
spectral flow) is the same as the change in $\# \M (X,g,\beta)$, we arrive at 
the main result of this paper:

\begin{thm}\label{T:main}
Let $X$ be a smooth oriented homology oriented $4$-mani\-fold with the integral 
homology of $S^1 \times S^3$, and define 
\[
\lambda_{\,\SW} (X) = \#\,\M(X,g,\beta) - w\,(X,g,\beta).
\]
Then $\lambda_{\,\SW} (X)$ is independent of the choice of metric $g$ and 
generic perturbation $\beta$; moreover, the reduction of $\lambda_{\,\SW} 
(X)$ modulo $2$ is the Rohlin invariant of $X$.
\end{thm}

When $X$ has the form $S^1 \times Y$ for an integral homology sphere $Y$, our 
invariant becomes the invariant of $Y$ studied by Weimin 
Chen~\cite{chen:casson} and Yuhan Lim~\cite{lim:sw3}. Indeed, our approach was 
inspired by their 
treatment of a similar issue with the count of irreducibles in the 
Seiberg--Witten moduli space on $Y$, namely, that this count can change in 
a one-dimensional family of parameters. To counter this, Chen and Lim 
(independently, following a suggestion of Kronheimer~\cite{donaldson:swsurvey})
introduced an index-theoretic, metric-dependent invariant of $Y$ that jumps 
in the same fashion as the above count. In fact, this `counter-term' is a 
combination of $\eta$--invariants associated to the Dirac and odd signature 
operators on $Y$. One can say that our correction term $w\,(X,g,\beta)$ plays 
the role of this combination of $\eta$--invariants in the setting when the
operators are periodic rather than $\R$--invariant over the end.

For an integral homology sphere $Y$, Lim~\cite{lim:swcasson} proved that the 
invariant of $Y$ defined in this fashion actually coincides,
up to an overall sign, with the Casson invariant $\lambda (Y)$. The latter 
is defined by counting either irreducible flat $SU(2)$ connections on $Y$ as 
in Taubes~\cite{taubes:casson}, or irreducible $SU(2)$ representations of its 
fundamental group; see Akbulut--McCarthy~\cite{akbulut-mccarthy}. The Casson 
invariant was extended by Furuta and Ohta~\cite{furuta-ohta} to an invariant 
$\lambda_{\,\FO}(X)$ of smooth $4$-manifolds $X$ with the $\Z[\Z]$-homology 
of $S^1 \times S^3$. This invariant counts irreducible flat $SU(2)$ 
connections on $X$ in the sense of the Donaldson gauge 
theory~\cite{donaldson-kronheimer}. Again, this count corresponds to a count 
of irreducible $SU(2)$ representations of $\pi_1 (X)$; in particular, the 
invariant vanishes for the homotopy $S^1 \times S^3$. The second and third 
authors studied 
$\lambda_{\,\FO}(X)$ in earlier papers~\cite{ruberman-saveliev:mappingtori,
ruberman-saveliev:survey}. Now we make the following conjecture, which we 
have verified in many examples.

\begin{theorem2}\label{sw=fo}
For any smooth oriented homology oriented $4$-mani\-fold $X$ with the 
$\Z[\Z]$-homology of $S^1 \times S^3$, one has 
\[
\lambda_{\,\SW} (X) = -\lambda_{\,\FO}(X).
\]
\end{theorem2}  

\noindent
If this conjecture turns out to be true, it will answer several long-standing 
questions in topology, the most striking of which is the vanishing of the 
Rohlin invariant for a homotopy $S^1 \times S^3$. 
See~\cite{ruberman-saveliev:survey} for a longer discussion.

Finally, we point out that Seiberg--Witten and Yang--Mills invariants for 
manifolds with $b^2_+ =0$ have also been considered by other authors.
Okonek--Teleman~\cite{okonek-teleman:b+0} and Teleman~\cite{teleman:definite} 
studied these in connection with problems about curves on complex surfaces 
of type ${\rm VII}_0$.  The topological setup in these papers is more general, 
in that $b^2 = b^2_{-} \neq 0$, and $b^1$ can be bigger than $1$. A related 
preprint of Lobb--Zentner~\cite{lobb-zentner:casson} treats the case of 
negative definite 4-manifolds with non-zero $b^2$ divisible by 4 and $b^1 = 1$, 
and defines an invariant by counting projectively flat $SU(2)$ connections; 
Zentner~\cite{zentner:casson-vanishing} has subsequently shown that this 
invariant actually vanishes. Recently, Fr{\o}yshov~\cite{froyshov:qhs-monopole}
has studied two numerical invariants of a negative definite $4$-manifold $X$ 
with $b^1 (X)=1$. One invariant is the Lefschetz number of a cobordism induced 
map on a version of the monopole Floer homology; its definition requires that 
$X$ has an embedded rational homology sphere $Y$ generating $H_3(X;\Z)$. The 
other invariant is the $h$--invariant, whose definition requires that either 
$b^2(X) = 0$ or a generator of $H_3(X;\Z)$ is represented by an integral 
homology sphere $Y$. It would be of interest to understand the relation of 
all of these works to the present paper. We show that our definition of the 
invariant $\lambda_{\,\SW} (X)$ can be extended to negative definite 
manifolds $X$ with $b^2_{-}(X) \neq 0$, at least under the assumption that 
either $b^2_{-}(X) =1$ or else there is a homology sphere $Y$ generating 
$H_3(X;\Z)$; this extension is described in Section~\ref{S:neg-def}.

Here is a brief outline of the paper.  We begin in Section~\ref{S:moduli} 
with a description of the blown-up Seiberg--Witten moduli space introduced 
in~\cite{KM} to deal with reducible solutions, and establish some 
transversality results for $1$-parameter moduli spaces that will be used 
later in the paper. In Section~\ref{S:correction}, we introduce the Dirac 
operator on periodic-end manifolds, and define the correction term 
$w\,(X,g,\beta)$. Section~\ref{S:fourier} contains the basic analysis of 
the Dirac operator on the infinite cyclic cover $\tilde X$ of a manifold 
$X$, including its Fredholm properties when acting on weighted Sobolev 
spaces $L^2_{k,\delta}\,(\tilde X)$. The main tool here is the 
Fourier--Laplace transform defined in~\cite{T} that relates this periodic 
Dirac operator on $\tilde X$ to a family of twisted Dirac operators on $X$
with parameter $z \in \C^*$. We show that the inverses of this family depend 
meromorphically on $z$, and describe the variation of the poles 
in a generic path of metrics and perturbations. In the following 
Sections~\ref{S:dirac-periodic} and~\ref{S:index-change}, we extend this 
analysis to the case of a manifold with a periodic end modeled on $\tilde X$, 
establishing asymptotics for solutions to the Dirac equation and a change of 
index formula. This formula is applied in Section~\ref{S:metric} to interpret 
the index change as a  spectral flow. Finally, in Section~\ref{S:invariant}, 
we put together the results of the preceding sections to match the change in 
$w\,(X,g,\beta)$ with jumps in $\#\M(X,g,\beta)$ in a generic $1$-parameter 
family of metrics and perturbations, and hence to prove the independence of 
$\lambda_{\,\SW} (X)$ from the choice of metric and generic perturbation.  
Section~\ref{S:rohlin} contains the proof that $\lambda_{\,\SW} (X)$ reduces 
modulo $2$ to the Rohlin invariant; this involves a careful choice of metric 
as in~\cite{RS} and perturbation that is equivariant with respect to the 
well-known involution in Seiberg--Witten theory. Section \ref{S:neg-def} 
describes an extension of $\lambda_{\,\SW} (X)$ to certain negative definite 
manifolds $X$ with $b^1(X) = 1$. The final Section~\ref{S:examples} provides 
some examples of explicit calculations; we remark that all of these are in 
agreement with Conjecture~\ref{sw=fo}.
\\[1ex]

\noindent
\textbf{Acknowledgments:} We thank Kim Fr{\o}yshov, Lev Kapitanski, Claude 
LeBrun, Leonid Parnovski,  and Cliff Taubes for useful discussions and sharing their 
expertise.  We also thank the referee for some perceptive comments and suggestions.


\section{Seiberg--Witten moduli spaces}\label{S:moduli}
Let $X$ be an oriented smooth homology $S^1 \times S^3$, and choose a spin 
structure on $X$. Two different spin structures on $X$ are equivalent as
\spinc structures hence our construction will be independent of this choice. 
Fix a homology orientation on $X$ by fixing a generator $1\in \Z = H^1(X;\Z)$. 
In this section, we will introduce, following~\cite{KM}, the blown-up 
Seiberg--Witten moduli space on $X$ and investigate its dependence on 
the (generic) metric and perturbation. 


\subsection{Definition}
Given a metric $g$ on $X$ and a form $\omega \in \Omega^2_+ (X,i\R)$, 
consider the triples $(A,s,\phi)$ comprised of a $U(1)$--connection $A$
on the determinant bundle of the spin bundle, a real number $s \ge 0$, 
and a positive spinor $\phi$ such that $\|\phi\|_{L^2} = 1$. We will 
slightly abuse notation by viewing $A$ as a form in $\Omega^1 (X,i\R)$
after fixing a trivialization of the determinant bundle. The gauge group 
$\G = \Map (X,S^1)$ acts freely on such triples by the rule $u (A,s,\phi) 
= (A - u^{-1}du,s,u \phi)$. The (perturbed, blown-up) \emph{Seiberg--Witten 
moduli space} consists of the gauge equivalence classes of the triples 
$(A,s,\phi)$ that solve the Seiberg--Witten equations
\begin{equation}\label{E:sw}
\begin{cases}
\; F_A^+ - s^2\,\tau(\phi) = \omega \\
\; D^+_A (X,g)\,(\phi) = 0.
\end{cases}
\end{equation}

The fact that $b^2_+(X) = 0$ allows for the following description of the 
forms $\omega \in \Omega^2_+ (X,i\R)$. We use the standard notation $\H^k$ 
for harmonic $k$-forms.

\begin{lemma}
For any $\omega \in \Omega^2_+ (X,i\R)$ there exists a unique $\beta \in
\Omega^1 (X,i\R)$ such that $d^+ \beta = \omega$, $d^*\beta = 0$, and 
$\beta$ is orthogonal to $\H^1 (X;i\R) \allowbreak \subset \Omega^1 (X,i\R)$.
\end{lemma}

\begin{proof}
Since $H^2_+ (X;i\R) = 0$, the Hodge decomposition for $\omega$ takes the 
form $\omega = d\alpha + d^*\gamma = d\alpha + *d\beta$ where $\beta = *
\gamma$. Since $*\omega = \omega$, we have $\omega = d\beta + *d\alpha$. 
By the uniqueness of the Hodge decomposition, $\alpha = \beta$ and hence
$\omega = d\beta + *d\beta = d^+ \beta$. One can of course choose $\beta
\in (\im d)^{\perp} = \ker d^*$, which makes it unique up to adding a 
harmonic 1-form in $\H^1 (X;i\R)$.
\end{proof}

In other words, we have a linear isomorphism $d^+: \P \to \Omega^2_+ (X,
i\R)$ with $\P = \ker d^*\,\cap\,\H^1 (X;i\R)^{\perp}$. The vector space 
$\P$ will be referred to as the \emph{space of perturbations}. Given 
a perturbation $\beta \in \P$, the Seiberg--Witten moduli space 
corresponding to $\omega = d^+ \beta$ in~\eqref{E:sw} will be denoted 
$\M (X,g,\beta)$. 


\subsection{Regularity}\label{S:regularity}
Fix an integer $k \ge 3$, and define the (blown up) configuration space 
$\widetilde \B$ as the Hilbert manifold of the $L^2_{k+1}$ gauge equivalence 
classes of triples $(A,s,\phi)$, where $A$ is a $U(1)$--connection on $X$ of 
Sobolev class $L^2_k$, $s$ is a real number, and $\phi$ is a positive $L^2_k$ 
spinor such that $\|\phi\|_{L^2} = 1$. The space $\widetilde\B$ has an 
involution that sends $(A,s,\phi)$ to $(A,-s,\phi)$. We define the 
configuration space $\B$ as the subset of $\widetilde \B$ where $s \ge 0$. 

Inside the configuration space $\widetilde\B$ sits the Hilbert submanifold
$\tmZ$ of the gauge equivalence classes of triples $(A,s,\phi)$ 
with $D^+_A (X,g)\,(\phi) = 0$; see~\cite[Lemma 27.1.1]{KM}.  Again, 
$\tmZ$ has the natural involution $(A,s,\phi)\to (A,-s,\phi)$. 
The condition $s \ge 0$ defines $\mZ \subset \B$, a Hilbert submanifold 
with boundary. The coordinate $s$ is a product factor, and the boundary 
$\p \mZ$ of $\mZ$ occurs at $s = 0$.

Let the map $\swz: \tmZ \to \Omega^2_+ (X,i\R)$ be given by $\swz (A,s,\phi) 
= F^+_A - s^2\,\tau (\phi)$, where $\Omega^2_+ (X,i\R)$ is completed in the
Sobolev $L^2_{k-1}$ norm. (This map is called $\varpi$ in~\cite{KM}). For 
any $\beta$ in the space $\P$ completed in the Sobolev $L^2_k$ norm, define
\[
\widetilde \M (X,g,\beta)\, =\, \swz^{-1}(d^+ \beta)\, \subset\, \tmZ.
\] 
The map $\swz$ is invariant with respect to the involution $(A,s,\phi) 
\to (A,-s,\phi)$, hence we have an induced involution on $\widetilde \M 
(X,g,\beta)$, and the Seiberg--Witten moduli space $\M (X,g,\beta)$ defined 
in the previous section is the intersection $\widetilde\M (X,g,\beta)\,\cap
\,\mZ$. Also define 
\[
\M^{0}(X,g,\beta)\, =\, \widetilde\M (X,g,\beta)\,\cap\,\p\mZ;
\]
the points in $\M^{0}(X,g,\beta)$ will be called \emph{reducible}.

\begin{proposition}\label{P:reg1}
For a generic $\beta \in \P$, the moduli space $\M (X,g,\beta) \allowbreak 
\subset \mathcal Z$ is a compact oriented manifold of dimension zero with empty 
$\M^{0}(X,g,\beta)$.
\end{proposition}

\begin{proof}
This is essentially contained in~\cite[Lemma 27.1.1]{KM}. The map $\swz:
\tmZ \to \Omega^2_+ (X,i\R)$ is Fredholm of index zero, and its 
restriction $\p\swz: \p\mZ \to \Omega^2_+ (X,i\R)$ is Fredholm of index 
$-1$. Therefore, for generic $\beta$,\, $\widetilde \M(X,g,\beta) = 
\swz^{-1}(d^+\beta)$ is a manifold of dimension zero and $\M^{0}(X,g,\beta) 
\allowbreak = (\p\swz)^{-1}(d^+\beta)$ is empty.  The moduli space 
$\M(X,g,\beta)$ is compact, and it is oriented as usual by the choice of 
orientation and homology orientation on $X$. 
\end{proof}

Any pair $(g,\beta)$ provided by Proposition~\ref{P:reg1} will be 
called \emph{regular}. Given a regular pair $(g,\beta)$, denote by 
$\#\,\M (X,g,\beta)$ the signed count of points in the (regular) moduli 
space $\M (X,g,\beta)$. 


\subsection{Regularity in families}\label{S:regular}
Let $g_I$ be a smooth path of metrics on $X$ parameterized by $I = [0,1]$, 
and consider the parameterized configuration space 
\[
\widetilde \B_I \; = \; \bigcup_{t \in I}\;\; \{t\}\times \widetilde \B_t,
\]
where $\widetilde \B_t$ stands for a copy of $\widetilde \B$ with respect to 
the metric $g_t$.  We will follow~\cite[page 479]{KM} and slightly  abuse 
notations in regarding $\widetilde \B_t$ as independent of $t$ and identifying 
$\widetilde \B_I$ with  $I \times \widetilde \B_0$. 

\begin{remark}\label{R:id}
The precise identification is as follows, cf.~\cite[Section 3.1]{RS}. For 
any $t \in I$, there is a unique automorphism $b_t: TX \to TX$ that is 
positive, symmetric with respect to $g_0$, and has the property that $g_0 
(u,v) = g_t(b_t(u),b_t(v))$. The map on orthonormal frames induced by 
$b_t$ gives rise to a map $\bar b_t: S_0^{\pm} \to S_t^{\pm}$ of spinor 
bundles associated with metrics $g_0$ and $g_t$. This map is an isomorphism 
preserving the fiberwise length of spinors. The identification $I \times 
\widetilde \B_0 \to \widetilde \B_I$ is then given by $(t,A,s,\phi) \to 
(t,A,s,\bar b_t\,(\phi))$.
\end{remark}

\begin{figure}[!ht]
\centering
\psfrag{Z}{$\mZ_I$}
\psfrag{dZ}{$\p\mZ_I$}
\psfrag{I}{$I$}
\psfrag{dM}{$\M_I$}
\includegraphics{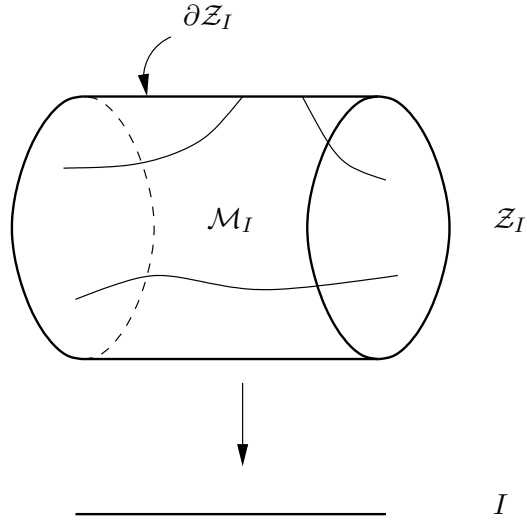}
\caption{Parameterized moduli space}
\label{fig:param}
\end{figure}

Let $\tmZ_t \subset \widetilde \B_t$ be given by $D^+_A(X,g_t)\,(\phi) = 0$.
The spaces $\tmZ_t$ may be assembled into a Hilbert manifold
\[
\tmZ_I\; =\; \bigcup_{t\in I}\;\;\{t\}\times \tmZ_t \subset 
\widetilde{\mathcal B}_I
\]
over $I$ with a natural projection $\pi: \tmZ_I \to I$. Both $\widetilde \B_I$ 
and $\tmZ_I$ admit an involution sending $(t,A,s,\phi)$ to $(t,A,-s,\phi)$. 
The conditions $s \ge 0$ and $s = 0$ define respectively subbundles $\mZ_I 
\subset \tmZ_I$ and $\p \mZ_I\subset \mZ_I$ over $I$; see 
Figure~\ref{fig:param}.

Let $\Omega_I$ be the subspace of $I\times\,\Omega^2 (X,i\R)$ comprised 
of the pairs $(t,\omega)$ such that $\omega$ is self-dual with respect 
to the metric $g_t$. Then we have a fibration $\pi: \Omega_I \to I$ 
that can be trivialized using the maps $b_t$ as in Remark~\ref{R:id}. 
Consider the commutative diagram

\begin{equation}\label{E:cd1}
\begin{CD}
\tmZ_I @ >\swz_I >> \Omega_I \\
@VV\pi V @ VV\pi V \\
I @> 1 >> I
\end{CD}
\end{equation}

\bigskip\noindent
where the map $\swz_I$ is given by $\swz_I\,(t,A,s,\phi) = (t, F^{+_t}_A 
- s^2\,\tau_t (\phi))$. The pre-image of a section $\omega_I = d^+\beta_I: 
I \to \Omega_I$ under $\swz_I$ is the parameterized moduli space 
\[
\widetilde \M_I\; =\; \bigcup_{t \in I}\;\; \{t\}\times \widetilde\M 
(X,g_t,\beta_t).
\] 

\noindent
Let us also consider
\[
\M_I = \widetilde \M_I\,\cap\,\mZ_I \;=\; \bigcup_{t \in I}\;\;\{t\}\times
\M (X,g_t,\beta_t)
\]
and 
\[
\M^{0}_I = \widetilde \M_I\,\cap\,\p \mZ_I \;=\; \bigcup_{t \in I}\;\;\{t\}
\times \M^{0}(X,g_t,\beta_t).
\]

\begin{theorem}\label{T:reg}
Let $g_I$ be a path of metrics and $\beta_0$, $\beta_1 \in \P$ a pair of 
perturbations such that $\M(X,g_0,\beta_0)$ and $\M (X,g_1,\beta_1)$ are 
regular. Then, for a generic path $\beta_t \in \P$ connecting $\beta_0$ 
to $\beta_1$, the parameterized moduli space $\M_I$ is a smooth oriented 
1-manifold with oriented boundary 
\[
\M(X,g_0,\beta_0)\,\cup\,\M(X,g_1,\beta_1)\,\cup\,\M^{0}_I.
\] 
In particular, 
\[
\#\,\M(X,g_1,\beta_1)\, -\, \#\M(X,g_0,\beta_0)\, =\, \#\,\M^{0}_I,
\]
where $\#\,\M^{0}_I$ stands for the signed count of points in $\M^{0}_I$.
\end{theorem}

\begin{proof}
Observe that $\chi_I$ is a Fredholm map of index zero, and choose an 
arbitrary section $\omega_I: I\to \Omega_I$ that restricts to $\omega_0 
= d^+\beta_0$ and $\omega_1 = d^+\beta_1$ at the endpoints of $I$. 
According to~\cite[Theorem 3.1]{smale:sard}, it can be approximated by a 
section, called again $\omega_I = d^+ \beta_I$, that is transversal to 
both $\p\swz_I: \p\mZ_I \to \Omega_I$ and $\swz_I: \tmZ_I \to \Omega_I$.
Then $\widetilde \M_I = \swz_I^{-1} (d^+\beta_I)$ is a regular oriented 
1-manifold with boundary $\widetilde \M (X,g_0,\beta_0)\,\cup\,
\widetilde \M(X,g_1,\beta_1)$, and $\M_I$ is a regular oriented 1-manifold 
whose boundary is as claimed.
\end{proof}

Any path $(g_I,\beta_I)$ with regular endpoints provided by 
Theorem~\ref{T:reg} will be called \emph{regular}. We will wish to 
calculate $\#\,\M^{0}_I$ and eventually, compare it with the jumps in 
the index of a certain differential operator; see Theorem~\ref{T:comparison}. 
The calculation of the index change will be greatly simplified if the regular path 
$(g_I,\beta_I)$ is chosen so that $g_I$ is constant near each $t\in I$ 
where $\M^{0}(X,g_t,\beta_t)$ is non-empty; such a path will be called 
\emph{special}. In the Appendix, we will prove the following result. 

\begin{theorem}\label{T:special}
Any regular path $(g_I,\beta_I)$ can be homotoped rel its endpoints to 
a special path.
\end{theorem}


\subsection{The reducibles}\label{S:reds}
It will be useful to have a more explicit description of $\M^{0}_I$, as 
much of this paper is concerned with the behavior of the Seiberg--Witten 
equations at reducible points. For any choice of $(g,\beta)$, the space 
$\M^{0}(X,g,\beta)$ consists of the gauge equivalence classes of pairs 
$(A,\phi)$ such that $F_A^+ = d^+ \beta$, $D_A^+ (X,g)\,(\phi) = 0$, and 
$\|\phi\|_{L^2} = 1$. The gauge equivalence classes $[A]$ of connections 
such that $F^+_A = d^+ \beta$ form the circle 
\begin{equation}\label{E:circle}
\beta\, +\, \H^1 (X;i\R)/H^1(X;i\Z).
\end{equation}
For each $[A]$ on this circle such that $\ker D^+_A(X,g) \ne 0$, solutions 
of the equation $D^+_A(X,g)(\phi) = 0$ with $\|\phi\|_{L^2} = 1$ form a 
sphere of dimension $2 \dim_{\C}\ker D^+_A(X,g) - 1$ and, after factoring 
out the residual $S^1$ gauge symmetry, the complex projective space of
(complex) dimension $\dim_{\C}\ker D^+_A(X,g) - 1$. 

In particular, for any regular $(g,\beta)$ the space $\M^0 (X,g,\beta)$ is 
empty, which means that $\ker D^+_A(X,g) = 0$ for all points $[A]$ on the 
circle~\eqref{E:circle}. Furthermore, for any pair $(g,\beta)$ in a regular 
path $(g_I,\beta_I)$ the space $\M^0 (X,g,\beta)$ is at most finite, which 
implies that $\ker D^+_A (X,g)\ne 0$ for at most finitely many points $[A]$ 
on the circle \eqref{E:circle}, and in fact that $\ker D^+_A(X,g) = \C$ at 
each of them.

Parametrize the circle \eqref{E:circle} by choosing a smooth function
$f: X \to S^1$ such that 
\begin{equation}\label{E:f}
[\,f^*(d\theta)] = 1\, \in\, H^1(X;\Z) 
\end{equation}
corresponding to our choice of homology orientation on $X$. Then the above 
discussion can be restated as follows in terms of the family 
\begin{equation}\label{E:Dz}
D^+_A (X,g) = D^+(X,g) + \beta - \ln z\cdot f^*(d\theta)\quad\text{with}\quad
|z| = 1.
\end{equation}

\begin{proposition}\label{P:kernels}
For any regular $(g,\beta)$, all operators in the family \eqref{E:Dz} 
have zero kernels. For any point $(g,\beta)$ in a regular path 
$(g_I,\beta_I)$, at most finitely many of the operators \eqref{E:Dz} have 
non-zero kernel, and that kernel is isomorphic to $\C$.
\end{proposition}


\section{The correction term}\label{S:correction}
Let $X$ be an oriented smooth spin homology $S^1 \times S^3$ with a fixed
homo\-logy orientation. In this section, for any regular pair $(g,\beta)$, 
we will introduce the correction term $w(X,g,\beta)$ and will prove that 
it is well defined. 


\subsection{End-periodic manifolds}\label{S:epmflds}
Let $Y \subset X$ be a connected spin 3-mani\-fold dual to the generator of 
$H^1(X;\Z)$ so that $X$ cut open along $Y$ is a spin homology cobordism $W$ 
from $Y$ to itself. Corresponding to the isomorphism $H_1 (X;\Z) \to \Z$ is 
the infinite cyclic cover 
\begin{equation}\label{E:cover}
\tilde X = \;\ldots \cup\;W_{-1}\;\cup\;W_0\;\cup W_1\;\cup \ldots  
\quad\text{with all}\quad W_n = W. 
\end{equation}
By an \emph{end--periodic (spin) manifold} whose end is modeled on $\tilde X$ 
we will mean any manifold $Z_+ = Z \cup \tilde X_+$, where $Z$ is a smooth 
compact spin manifold with spin boundary $\p Z = Y$, and
\[ 
\tilde X_+ = \; W_0\;\cup\, W_1\;\cup\, W_2\;\cup \ldots
\quad\text{with all}\quad W_n = W. 
\] 
Any metric $g$ on $X$ naturally lifts to a periodic metric on $\tilde X_+$ 
and extends to an end--periodic metric on $Z_+$, again called $g$. The same 
holds for the perturbation 1--forms $\beta \in \P$.
 

\subsection{End-periodic Dirac operators}
Given an end-periodic manifold $Z_+$, with any choice of metric $g$, 
perturbation 1-form $\beta$, and spin structure as above one associates 
the perturbed Dirac operator 
\begin{equation}\label{E:dirac}
D^+(Z_+,g,\beta):= D^+ (Z_+,g) + \beta:\; L^2_1\,(Z_+,S^+) \to L^2 (Z_+,S^-).
\end{equation}

\begin{theorem}\label{T:taubes2}
The operator \eqref{E:dirac} is Fredholm for any regular pair $(g,\beta)$ on 
the manifold $X$.
\end{theorem}

\begin{proof}
The operator $D^+(X,g) + \beta$ has zero index. Therefore, according to 
Taubes~\cite[Lemma 4.3]{T}, the operator \eqref{E:dirac} is Fredholm if
and only if the kernels of the operators in the family \eqref{E:Dz} 
vanish for any choice of smooth function $f: X \to S^1$ satisfying 
\eqref{E:f}. But for a regular $(g,\beta)$, these kernels vanish by 
Proposition~\ref{P:kernels}. Note that the statement of Lemma 4.3 in 
\cite{T} should read $\delta \in \R$ rather than $\delta \in \R^*$ 
(interpreted to mean $\R \setminus \{0\}$), which can be confirmed by 
a reading of the proof of that lemma. This typo was pointed out to us 
by the referee.
\end{proof}

In Sections~\ref{S:fourier} and~\ref{S:dirac-periodic} below, we 
will explain in detail the basic analytical machinery behind Theorem 
\ref{T:taubes2} and re-prove it in more general situation; see Corollary 
\ref{C:taubes}. 


\subsection{Definition of the correction term}
Let $Z_+ = Z \cup \tilde X_+$ be an end-periodic manifold and $(g,\beta)$ 
a regular pair on $X$ so that the operator \eqref{E:dirac} is Fredholm by 
Theorem \ref{T:taubes2}. Define the \emph{correction term}
\[
w\,(X,g,\beta)\, =\, \ind_{\,\C} D^+ (Z_+,g,\beta)\, + \,\sign (Z)/8.
\]
Note that in general $w\,(X,g,\beta)$ will be a rational number.

\begin{proposition}\label{P:3.2}
The correction term $w\,(X,g,\beta)$ is independent of the choices of $Z$ 
and $Y \subset X$, and the way $g$, $\beta$, and the spin structure are 
extended to $Z$. 
\end{proposition}

\begin{proof}
Given two choices, $Z$ and $Z'$ with $\p Z = \p Z' = Y$, we use the 
excision principle -- see Section~\ref{S:excision} -- and the index 
theorem to obtain
\begin{multline}\notag
\ind_{\,\C} D^+ (Z'_+,g,\beta) - \ind_{\,\C} D^+ (Z_+,g,\beta) =
\ind_{\,\C} D^+ (- Z \cup Z')\\ = - \frac 1 {24}\,\int_{- Z \cup Z'} p_1
= - \frac 1 8\,\sign (-Z \cup Z') = \frac 1 8\,\sign (Z) - \frac 1 8\,
\sign (Z'),
\end{multline} 
which proves that $w(X,g,\beta)$ is independent of the choices of $Z$ and
the extensions.  

Next, let $Y$ and $Y'$ be two choices for cutting $X$ open, and choose a 
lift of each to the infinite cyclic cover $\tilde X$. Call the lifts $Y$
and $Y'$ again. Since $Y$ and $Y'$ are both compact, one can apply the 
covering translation to $Y'$ sufficiently many times to ensure that $Y$ 
and $Y'$ are disjoint. Both $Y$ and $Y'$ separate $\tilde X$ hence they 
become the boundary components of a spin cobordism $V \subset \tilde X$.
We claim that the signature of $V$ is zero. This fact is well-known, but 
we do not know a precise reference; compare~\cite{ruberman:ds} 
and~\cite{ogasa:ribbon-mu}.

To verify the claim, note that $V$ is disjoint from its image under a 
sufficiently high \emph{prime} power $p$ of the covering translation, and 
hence the projection of $V$ into the $p$-fold cyclic cover of $X$ is an 
embedding.  But a standard argument in knot 
theory~\cite[\S 5]{gordon:aspects} says that this $p$-fold cover is a 
rational homology $S^1 \times S^3$, so that the intersection form on its 
2-dimensional rational homology vanishes identically.  It follows that 
the intersection form on $V$ vanishes as well, so the signature of $V$ 
is zero.  That $w (X,g,\beta)$ is independent of the choice of $Y$ now 
follows by the argument in the previous paragraph applied to $Z' = 
Z \cup V$.
\end{proof}


\section{Dirac operators on infinite cyclic covers}\label{S:fourier}
In the definition of the correction term $w\,(X,g,\beta)$ in Section
\ref{S:correction}, we were able to avoid most of the analysis on 
end-periodic manifolds by simply quoting Theorem \ref{T:taubes2}, 
which is essentially due to Taubes \cite{T}. Getting a firm grip on 
this analysis and, in particular, on the ideas involved in the proof 
of Theorem \ref{T:taubes2}, becomes essential once we wish to prove 
that the difference $\#\M (X,g,\beta) - w (X,g,\beta)$ is independent 
of the choice of regular $(g,\beta)$. We will study this analysis in 
the next four sections. 

This section is dedicated to the Fourier--Laplace transform (following 
Taubes~\cite{T}) and the role it plays in establishing Fredholmness of 
the Dirac operator on the infinite cyclic cover $\tilde X \to X$.  The 
same transform arises in related contexts; for example the terminology 
`$Z$--transform' is widely used in the literature 
(e.g.~\cite{jury:z-transform}). A number of proofs in this section are 
standard, and are largely omitted.

 
\subsection{Infinite cyclic cover}\label{S:furl} 
Let $X$ be an oriented smooth spin homology $S^1 \times S^3$ with a fixed 
homology orientation, and $f: X \to S^1$ a smooth function satisfying 
\eqref{E:f}. Let $\tilde X$ be the infinite cyclic cover as in \eqref{E:cover}
and $\tilde f: \tilde X \to \R$ a lift of $f$ such that $\tilde f (x + n) = 
\tilde f(x) + n$. Here, $x \mapsto x + 1$ stands for the covering translation 
$\tilde X \to \tilde X$. Any metric $g$ on $X$ lifts to a periodic metric on 
$\tilde X$ again called $g$, and the same holds for the spin structure.

Note that in the product case, when $X = S^1 \times Y$ with the product 
metric and spin structure, and $Y$ is an integral homology 3-sphere, 
one has $\tilde X = \R \times Y$. The function $\tilde f: \R \times Y 
\to \R$ can then be chosen to be $\tilde f (\theta,y) = \theta$.  

 
\subsection{Definition of the Fourier--Laplace transform} 
Given a spinor $u \in C^{\infty}_0 (\tilde X,S^{\pm})$ and a complex number 
$\mu \in \C$, the \emph{Fourier--Laplace transform} of $u$ is defined as 
\[ 
\hat u_{\mu}(x) = e^{\,\mu \tilde f(x)}\,\sum_{n = -\infty}^{\infty}  
e^{\,\mu n} u(x+n), 
\] 
where $x+n$ denotes the result of applying to $x \in \tilde X$ the 
covering translation $n$ times. Since $u$ has compact support, the 
above sum is finite. One can easily check that $\hat u_{\mu}(x + 1) =  
\hat u_{\mu} (x)$ for all $x \in \tilde X$. Therefore, for every  
$\mu\in \C$, we have a well defined spinor $\hat u_{\mu}$ over $X$.  
Note that the spinor $\hat u_{\mu}$ depends analytically on $\mu$,  
and that for any $x\in X$, we have  
\begin{equation}\label{E:2pi} 
\hat u_{\mu + 2\pi i}\,(x) = e^{\,2\pi i\, f(x)}\,\hat u_{\mu}(x). 
\end{equation}

In order to recover $u$ from its Fourier--Laplace transform, we do not 
need to know $\hat u_{\mu}$ for all $\mu \in \C$; in fact, it suffices 
to know $\hat u_{\mu}$ for $\mu$ in just one interval of the form 
$I(\nu) = \{\,\nu + i\theta\;|\; 0\le \theta \le 2\pi\,\}$ with $\nu 
\in \C$. The formula is as follows\,: 
\begin{equation}\label{E:zinv}
u(x + n) = \frac 1 {2\pi i}\,\int_{I(\nu)}\; e^{-\mu (f(x) + n)}\, 
\hat u_{\mu}(x)\,d\mu, 
\end{equation} 
where $x \in W_0$ and $n$ is an arbitrary integer. This can be checked by 
direct substitution. 

 
\subsection{The Fourier-Laplace transform in weighted Sobolev spaces} 
Given $\delta \in \R$ and an integer $k \ge 0$, we will say that $u \in 
L^2_{k,\delta}\, (\tilde X,S^{\pm})$ if and only if $e^{\delta\tilde f}\,u \in  
L^2_k\,(\tilde X,S^{\pm})$, and let
\[ 
\| u \|_{L^2_{k,\delta}\,(\tilde X,S^{\pm})} =  
\|\,e^{\delta\tilde f}\,u \|_{L^2_k\,(\tilde X,S^{\pm})}.  
\] 
We wish to extend our definition of the Fourier--Laplace transform to these 
weighted Sobolev spaces. For the sake of brevity, we will often omit $S^{\pm}$ 
from our notations and write $L^2\,(\tilde X)$ for $L^2\,(\tilde X,S^{\pm})$, 
etc. The basic observations are summarized as follows.

\begin{proposition}\label{P:reed1}
For any $u\in C_0^{\infty}(\tilde X)$ and $\nu \in \C$, the restriction  
of the family $\hat u_{\mu}$ to the interval $I(\nu)$ belongs to the Hilbert 
space $L^2 (I(\nu),L^2_k (X))$. Moreover, there is a constant $C$ such that 
\[
\|\,\hat u_{\mu}|_{I(\nu)}\|^2_{\,L^2(I(\nu),L^2_k (X))}\; \le\;
C\cdot\|u\|^{\,2}_{\,L^2_{k,\delta}\,(\tilde X)}
\quad \text{with}\quad \delta = \Re\,(\nu).
\]
\end{proposition}

Proposition~\ref{P:reed1} is proved using standard arguments of Fourier 
analysis (see for instance~\cite[page 290]{reed-simon}) that 
readily extend to the end-periodic case. It shows that, for any $\nu 
\in \C$ with $\Re\,(\nu) = \delta \in \R$, the assignment $u \mapsto 
\hat u_{\mu}|_{I(\nu)}$ can be uniquely extended to bounded linear 
operators  
\begin{equation}\label{E:Az} 
A(\nu): L^2_{k,\delta} (\tilde X) \to L^2(I(\nu),L^2_k(X)),\quad k 
\in \Z_+. 
\end{equation} 
 The following proposition is proved using formula \eqref{E:zinv} and the 
Parseval relation; see~\cite[page 290]{reed-simon} and~\cite[Lemma 2]{nazarov} 
in the product case. 

\begin{proposition} 
For any $\nu \in \C$ the operators~\eqref{E:Az} are linear isomorphisms.
\end{proposition} 
 
Finally, for use in Section~\ref{S:residues}, it will be helpful to know 
that the Fourier--Laplace transform of a spinor in a weighted Sobolev space 
is holomorphic in a specific region $V \subset \C$. We use the term 
{\em holomorphic} to mean that the function $V \to L^2_k(X)$ that assigns 
$\hat u_{\mu}$ to $\mu$ can be expressed as a power series in $\mu$ convergent 
in the $L^2_k(X)$ norm. There are many possible statements along these lines; 
we will content ourselves with the following result.

\begin{lemma}\label{L:half-plane}
Suppose that $u \in L^2_\delta\,(\tilde X)$ is a smooth spinor. If $u$ has 
support in $\tilde X_+ = W_0\cup W_1\cup\ldots$ then $\hat u_{\mu}$ is 
holomorphic in the half plane $\Re \mu < \delta$.
\end{lemma}

\begin{proof} 
It suffices to consider the case of a smooth spinor $u \in L^2 (\tilde X)$ 
with support in $\tilde X_+$. To show that $\hat u_{\mu}$ is holomorphic in 
the half plane $\Re \mu < 0$, all we need to do is estimate, for $x\in W_0$, 
the $L^2$--norm of the tail 
\[
\sum_{n \ge N}\; e^{\mu n}\, u(x + n).
\]
The latter norm can be estimated by 
\[
\left(\sum\; | e^{2\mu n}|\right)^{1/2} \cdot 
\left(\sum\; \|u(x+n)\|^2_{L^2}\right)^{1/2}
\]
using the H\"{o}lder inequality. The first series here converges as 
long as $\Re \mu < 0$ and approaches 0 as $N \to \infty$, and the second 
is estimated from above by $\| u \|_{L^2(\tilde X)}$.  
\end{proof}
 
 
\subsection{The Fourier--Laplace transform of perturbed Dirac operators} 
Let $\beta \in \P$ be a perturbation 1--form pulled back to $\tilde X$. 
We will be interested in the Fredholm properties of the perturbed Dirac 
operator
\begin{equation}\label{E:Dtilde}
D^+ (\tilde X,g,\beta): =  D^+ (\tilde X,g) + \beta: \;
L^2_{1,\delta}\,(\tilde X,S^+) \to L^2_{\delta}\,(\tilde X,S^-)
\end{equation}
in weighted Sobolev spaces. The  Fourier--Laplace transform of such an 
operator is the family of operators, parameterized by $\mu \in \C$, 
obtained by conjugating the operator by the Fourier--Laplace transform. 
A straightforward calculation shows that the Fourier--Laplace transform
of $D^+ (\tilde X,g,\beta)$ is the holomorphic family 
\begin{equation}\label{E:family}
D^+_{\mu} (X,g,\beta) = D^+ (X,g,\beta) - \mu\cdot f^*(d\theta)
\end{equation}
of perturbed Dirac operators on $X$ with $\mu \in \C$ (compare with 
\eqref{E:Dz}). Here, we use notation 
\[
D^+ (X,g,\beta): = D^+ (X,g) + \beta
\] 
for the perturbed Dirac operator on $X$. Similar formulas hold for the 
full Dirac operator and for the negative chiral Dirac operator.

\begin{proposition}\label{P:taubes1}  
The operator \eqref{E:Dtilde} is Fredholm if and only if the operators 
$D^+_{\mu}(X,g,\beta)$ are invertible for all $\mu$ with $\Re \mu = 
\delta$.
\end{proposition}

\begin{proof}
Since $D^+ (X,g,\beta)$ is an elliptic operator of index zero, the statement 
follows from Taubes~\cite[Lemma 4.3]{T}.
\end{proof}

Therefore, to understand the index theory of the perturbed Dirac operator 
\eqref{E:Dtilde}, we need to study the family \eqref{E:family}. The subset 
$\Ss(g,\beta)$ of the complex plane consisting of all $\mu \in \C$ for 
which $D^+_{\mu}(X,g,\beta)$ is not invertible will be called the 
\emph{spectral set} of the family $D^+_{\mu}(X,g,\beta)$.

\begin{lemma}
The spectral set $\Ss(g,\beta)$ is independent of the choice of the 
function $f: X \to S^1$.
\end{lemma}

\begin{proof}
For any two choices of $f$, the 1--forms $f^*(d\theta)$ differ by an 
exact form $dh$, where $h: X\to \R$ is a smooth function. The result 
now follows from the easily verified formula $
D^+_{\mu} (X,g,\beta) - \mu\,dh = e^{\mu h} D^+_{\mu} (X,g,\beta)\, 
e^{-\mu h}$.
\end{proof}

\begin{theorem}\label{T:mero}
Let $(g_I,\beta_I)$ be a regular path. For any $t \in I$, the 
spect\-ral set $\Ss (g_t,\beta_t)$ is a discrete subset of the complex 
plane, and the inverse of $D^+_{\mu}(X,g_t,\beta_t)$ is a meromorphic 
function of $\mu \in \C$.
\end{theorem}

\begin{proof}
According to Proposition~\ref{P:kernels}, the spectral set $\Ss(g_t,\beta_t)$ 
is a proper subset of $\C$ for any $t \in I$. Having noted this, abstract out 
the salient features of our situation and view $D^+_{\mu}(X,g_t,\beta_t)$ for 
any fixed $t \in I$ as a family of the shape
\[
T + \mu A :L^2_1\,(X,S^+) \to L^2\,(X,S^-),
\]
where $T$ (the operator $D^+ (X,g_t,\beta_t)$ in our case) is an index zero 
Fredholm operator, and $A$ (Clifford multiplication by 
$- f_t^*(d\theta)$) is a compact operator. Fix $\mu_0$ such that the 
operator $T + \mu_0 A$ is invertible.

Consider the operator $(T + \mu A)(T + \mu_0 A)^{-1}: L^2(X,S^-) \to 
L^2(X,S^-)$. This is a bounded operator, and we can rewrite it as
\[
I + (\mu - \mu_0) A (T + \mu_0 A)^{-1} = I + (\mu - \mu_0)\,K.
\]
The operator $K = A(T + \mu_0 A)^{-1}: L^2(X,S^-) \to L^2(X,S^-)$ is compact 
since both $A: L^2\to L^2$ and $(T + \mu_0 A)^{-1}: L^2\to L^2_1$ are bounded 
so that their composition factors through the compact embedding $L^2_1 \to 
L^2$. Thus we can apply the spectral theory of compact operators to the study 
of our family. For $\mu \ne \mu_0$ we conclude that $T + \mu A$ is invertible 
if and only if $\zeta = -(\mu - \mu_0)^{-1}$ is not in the spectrum of $K$. 
The spectrum of $K$ is a compact subset $\Spec(K)$ of the complex plane with 
only $0$ as an accumulation point. Thus the spectral set is discrete. 

Furthermore, the resolvent $(K - \zeta I)^{-1}$ of a compact operator is 
meromorphic in $\zeta = - 1/(\mu - \mu_0)$ away from $\zeta = 0$ hence the 
inverse of $T + \mu A$ is meromorphic in $\mu \in \C$.
\end{proof}

\begin{corollary}
Let $(g_I,\beta_I)$ be a regular path.  For any $t \in I$, the operator 
$D^+(\tilde X,g_t,\beta_t): L^2_{1,\delta}\,(\tilde X,S^+) \to L^2_{\delta}
\,(\tilde X,S^-)$ is Fredholm for all but a discrete set of $\delta \in \R$ 
with no accumulation points.
\end{corollary}

The set of $\delta$'s for which the above operator fails to be Fredholm may 
well depend on $t$; this dependence is examined in more detail in the next 
section.


\subsection{Spectral set as a function of $t$}
Given a path $(g_I,\beta_I)$ of metrics and perturbations on $X$ consider 
the \emph{parameterized spectral set}
\begin{equation}\label{E:Ss}
\Ss_I \,=\, \bigcup_{t \in I}\; \Ss (g_t,\beta_t)\,\subset\, \C.
\end{equation}
Let $t \in I$ be such that the spectral set $\Ss(g_t,\beta_t)$ is non-empty 
and, for any $\mu_j \in \Ss(g_t,\beta_t)$, consider the operator 

\begin{equation}\label{E:proj}
P_{\mu_j} = \frac 1 {2\pi i}\oint_{\Gamma}\; 
(D^+ (X,g_t,\beta_t) - \mu\cdot f^*_t (d\theta))^{-1}\,d\mu,
\end{equation}

\medskip\noindent
where $\Gamma$ is a small loop in the $\mu$--plane encircling $\mu_j$ once 
in the positive direction. 

\begin{theorem}\label{T:spectra}
Let $(g_I,\beta_I)$ be a regular path and suppose that $\mu_j \in 
\Ss(g_t,\beta_t)$ is such that the rank of the operator \eqref{E:proj} 
is one. Then there exist an open neighborhood $U (\mu_j)$ and a 
real $\ep > 0$ such that the intersection
\[
\bigcup_{|t - s| < \ep}\;\Ss (g_s,\beta_s)\,\cap\,U(\mu_j)
\]
is an embedded curve.
\end{theorem}

\begin{proof}
We will use notations from the proof of Theorem~\ref{T:mero} and recall 
from that proof that the spectral set of $T + \mu A$ coincides, up to 
shift and inversion, with the spectrum of the compact operator $K$. The 
operator \eqref{E:proj} then takes the form 
\[
P_{\mu_j} = \frac 1 {2\pi i}\oint_{\Gamma}\;(T + \mu A)^{-1}\,d\mu.
\]

\medskip\noindent
Write $(T + \mu A)^{-1} = - (T + \mu_0 A)^{-1}\,\zeta (K - \zeta I)^{-1}$
with $\zeta = -(\mu - \mu_0)^{-1}$ then, after changing coordinates,
\[
P_{\mu_j} = (T + \mu_0 A)^{-1}\;\frac 1 {2\pi i} \oint_{\Gamma'}\;\zeta^{-1}
(K - \zeta I)^{-1}\,d\zeta = (T + \mu_0 A)^{-1}\zeta_j^{-1}\,\Pi_{\zeta_j},
\]
where $\Gamma'$ is a small loop in the $\zeta$--plane encircling $\zeta_j = 
-(\mu_j - \mu_0)^{-1}$ once in the positive direction, and $\Pi_{\zeta_j}$ 
is the projector onto the generalized eigenspace of $K$ corresponding to 
$\zeta_j$. Then $\rk\,(\Pi_{\zeta_j}) = \rk\,(P_{\mu_j}) = 1$, and the 
result follows from the perturbation theory of compact 
operators~\cite[Theorem VII.1.8]{K}
\end{proof}

\begin{remark}
We will show later in Proposition~\ref{P:gen} that, for any special path 
$(g_I,\beta_I)$, the condition of Theorem~\ref{T:spectra} on the rank of 
the operator \eqref{E:proj} is automatically satisfied for all $\mu_j 
\in i \R\,\cap\,\Ss_I$.
\end{remark}


\section{Dirac operators on end--periodic manifolds}\label{S:dirac-periodic}
In this section, we extend the results obtained in Section~\ref{S:fourier} 
for Dirac operators on infinite cyclic covers to Dirac operators on general 
manifolds with periodic ends. Taubes' paper~\cite{T} is again the basic 
reference; a rather different geometric application of end-periodic 
operators may be found in~\cite{mazzeo-pollack-uhlenbeck:yamabe}.

 
\subsection{Weighted Sobolev spaces} 
Let $Z_+$ be an end-periodic manifold as defined in Section~\ref{S:epmflds}, 
and $f: X \to S^1$ a smooth function satisfying \eqref{E:f} lifted to a 
function $\tilde f: \tilde X \to \R$ as in Section~\ref{S:furl}. Given 
$\delta \in \R$, extend the function $\delta \cdot \tilde f(x): \tilde X_+ 
\to \R$ to a smooth function $h: Z_+ \to \R$. We will say that $\varphi
\in L^2_{k,\delta}\,(Z_+,S^{\pm})$ if and only if $e^h\,\varphi\in L^2_k\,
(Z_+,S^{\pm})$, with  
\[ 
\|\varphi\|_{L^2_{k,\delta}\,(Z_+,S^{\pm})} = \|e^h\,\varphi\|_{L^2_k 
\,(Z_+,S^{\pm})}.  
\] 
Note that different extensions $h$ of the same function $\delta\cdot 
\tilde f(x): \tilde X_+ \to \R$ give equivalent norms on $L^2_{k,\delta} 
\,(Z_+,S^{\pm})$. Also note that the maps $L^2_k\,(Z_+,S^{\pm}) \to  
L^2_{k,\delta}\,(Z_+,S^{\pm})$ sending $\varphi$ to $e^{h}\varphi$ are 
isomorphisms. 
 
 
\subsection{End-periodic Dirac operators} 
Given an end-periodic manifold $Z_+$, with any choice of pair $(g,\beta)$ 
one associates the perturbed Dirac 
operator $D^+(Z_+,g,\beta): C^{\infty} (Z_+,S^+) \to C^{\infty} (Z_+,S^-)$. 
The closures of these operators with respect to the weighted Sobolev 
$L^2$--norms,
\begin{equation}\label{E:taubes}
D^+ (Z_+,g,\beta): L^2_{1,\delta}\,(Z_+,S^+) \to L^2_{\delta}\,(Z_+,S^-),
\end{equation}
compare with \eqref{E:dirac}, are related by the commutative diagram  

\[ 
\begin{CD} 
L^2_{1,\delta}\,(Z_+,S^+) @> D^+ >> L^2_{\delta}\,(Z_+,S^-) \\ 
@VV e^h V @VV e^h V \\ 
L^2_1 (Z_+,S^+) @> D^+ - dh >> L^2\,(Z_+,S^-) 
\end{CD} 
\] 

\medskip\noindent
whose vertical arrows are isomorphisms.  In particular, over the end  
$\tilde X_+$, the operators $D^+ = D^+ (\tilde X_+,g,\beta)$ and $D^+ 
- dh = D^+ (\tilde X_+,g,\beta) - \delta f^*(d\theta)$ are intertwined 
by these isomorphisms.   

\begin{proposition}\label{P:taubes}
The operator \eqref{E:taubes} is Fredholm if and only if the operators 
$D^+_{\mu}(X,g,\beta)$ are invertible for all $\mu$ with $\Re \mu = \delta$.
\end{proposition}

\begin{proof}
According to Taubes~\cite[Proposition 4.1]{T}, it is sufficient to show 
that the statement holds for the operator \eqref{E:Dtilde}.  The latter 
was the subject of Proposition~\ref{P:taubes1}. 
\end{proof}

\begin{corollary}\label{C:taubes}
Let $(g_I,\beta_I)$ be a regular path. For any $t \in I$, the operator 
$D^+(Z_+,g_t,\beta_t): L^2_{1,\delta}\,(Z_+,S^+)\to L^2_{\delta}\,(Z_+,S^-)$ 
is Fredholm for all but a discrete set of $\delta \in \R$ with no 
accumulation points. This set may depend on $t$ but, for a fixed $t$, it 
is independent of the way the metric, perturbation 1--form, and the spin
structure are extended to $Z$. 
\end{corollary} 

For the rest of this section, we will assume that the pair $(g,\beta)$ 
belongs to a regular path even though it need not be regular itself. 
Let $\delta \in \R$ be such that the operator~\eqref{E:taubes} is 
Fredholm, and denote its index by $\ind_{\delta} D^+(Z_+,g,\beta)$. Our 
study of this index will require the excision principle, which we will 
explain next.
 

\subsection{The excision principle}\label{S:excision}
The observation that the excision principle for operators on compact 
manifolds~\cite{atiyah-singer:I} (compare~\cite{donaldson-kronheimer}) 
extends to the non-compact setting is due to Gromov and 
Lawson~\cite{gromov-lawson:complete}.   A nice exposition of the 
non-compact version is in Charbonneau's thesis~\cite{char}. 

Let $A_1$, $B_1$, $A_2$, and $B_2$ be (not necessarily compact) oriented 
4-manifolds such that  $\p A_1 = \p A_2 = Y$ and $\p B_1 = \p B_2 = -Y$, 
where $Y$ is a compact oriented 3-manifold. Let  
operators  
\begin{gather} 
D_1: L^2 (A_1\cup B_1)\to L^2 (A_1\cup B_1) \notag \\ 
D_2: L^2 (A_2\cup B_2)\to L^2 (A_2\cup B_2) \notag 
\end{gather} 
be (unbounded) Fredholm differential operators such that $D_1 = D_2$  
on $Y$. In our applications, $D_1$ and $D_2$ will be Dirac operators  
plus perhaps zero order terms. Suppose that  
\begin{gather} 
\bar D_1: L^2 (A_1\cup B_2)\to L^2 (A_1\cup B_2) \notag \\ 
\bar D_2: L^2 (A_2\cup B_1)\to L^2 (A_2\cup B_1) \notag 
\end{gather} 
defined as 
\[ 
\bar D_1 = 
\begin{cases} 
\; D_1 &\;\text{on\; $A_1$} \\ 
\; D_2 &\;\text{on\; $B_2$} 
\end{cases} 
\qquad\text{and}\qquad 
\bar D_2 = 
\begin{cases} 
\; D_2 &\;\text{on\; $A_2$} \\ 
\; D_1 &\;\text{on\; $B_1$} 
\end{cases} 
\] 
are (unbounded) Fredholm differential operators. Then  
\[ 
\ind D_1 + \ind D_2 = \ind \bar D_1 + \ind \bar D_2. 
\] 
  

\subsection{Properties of the index} 
Our first application of the excision principle will be to show that the 
index of the operator~\eqref{E:taubes} is well defined.  A point 
that we will use without further comment is that Clifford multiplication 
by a real $1$-form such as $dh$ is a skew-adjoint operator on spinors so 
that multiplication by a pure imaginary 1-form like $\beta$ is then 
self-adjoint.

\begin{proposition} 
Let $\delta \in \R$ be such that the operator~\eqref{E:taubes} is Fredholm. 
Then its index $\ind_{\delta} D^+(Z_+,g,\beta)$ is independent of the ways 
in which the metric, spin structure, perturbation form, and function $\delta 
\cdot \tilde f$ are extended to $Z$.  
\end{proposition} 
 
\begin{proof} 
Consider two different extensions and apply the excision principle to the 
operators  
\begin{align*} 
D_1 = D^+ (Z_+,g_1) + \btilde_1 - dh_1\quad &\text{on}\quad Z_+ = 
Z \cup \tilde X_+ \quad\text{and} \\ 
D_2 = D^+ (-Z_+,g_2) + \btilde_2+ dh_2\quad &\text{on}\quad -Z_+ = 
-\tilde X_+ \cup (-Z), 
\end{align*} 
both of which are Fredholm. First, observe that  
\begin{multline}\notag
D^+ (-Z_+,g_2) + \btilde_2 + dh_2
= D^- (Z_+,g_2) + \btilde_2 + dh_2 \\
= (D^+ (Z_+,g_2) + \btilde_2 - dh_2)^*. 
\end{multline}
Therefore, $\ind D_2 = -\ind_{\delta}(D^+ (Z_+,g_2) + \btilde_2)$. Second, 
the operator $\bar D_1$ is (up to zero order terms) the Dirac operator  
$D^+$ on the compact manifold $Z \cup (-Z)$. In particular, $\ind  
\bar D_1 = 0$. Finally, the manifold $-\tilde X_+ \cup \tilde X_+$  
admits an orientation reversing involution that takes the operator  
$\bar D_2$ to its adjoint, therefore, $\ind \bar D_2 = 0$. The  
excision principle now reads  
\[ 
\ind_{\delta}(D^+ (Z_+,g_1) + \btilde_1) - \ind_{\delta}(D^+ (Z_+,g_2) +
\btilde_2) = \ind \bar D_1 + \ind \bar D_2 = 0, 
\] 
which completes the proof.  
\end{proof} 
 
The next result will be helpful later when we compare the indices
$\ind_{\delta} D^+(Z_+,g,\beta)$ for different values of $\delta$. Given  
an end-periodic manifold $Z_+ = Z \cup \tilde X_+$, consider the end  
periodic manifold $Z^*_+ = Z \cup (-\tilde X_-)$, where  
\[ 
\tilde X_-\; =\; \ldots \cup\;W_{-2}\;\cup\;W_{-1}\; =\; \tilde X -  
\tilde X_+. 
\] 
Note that this construction corresponds to the change of homology  
orientation on $X$. Respectively, the function $\tilde f$ is replaced  
by $-\tilde f$, and its extension $h: Z_+\to \R$ by $-h: Z_+^*\to \R$.
Note, however, that $\btilde$ and its extension are unchanged. 
 
\begin{proposition} 
The operator $D^+(Z_+,g,\beta)$ is Fredholm if and only if $D^+(Z_+^*,g,\beta)$ 
is Fredholm, and $\ind_{\delta} D^+ (Z_+,g,\beta) = \ind_{\delta} D^+ 
(Z_+^*,g,\beta)$. 
\end{proposition} 
 
\begin{proof} 
Apply the excision principle to the operators $D_1 = D^+(Z_+,g) \allowbreak + 
\btilde - dh$ and $D_2 = D^+ (-Z^*_+,g) + \btilde - dh$, which are 
Fredholm. Since  
\[ 
D^+(-Z^*_+,g) + \btilde - dh = D^-(Z_+^*,g) + \btilde - dh = 
(D^+(Z_+^*,g) + \btilde + dh)^*,  
\] 
we conclude that $\ind D_2 = - \ind_{\delta}D^+ (Z_+^*,g,\beta)$. The  
excision principle then tells us that 
\[ 
\ind_{\delta}D^+(Z_+,g,\beta) - \ind_{\delta}D^+ (Z_+^*,g,\beta) = 
\ind \bar D_1 + \ind \bar D_2.  
\] 
The operator $\bar D_1$ is (up to zero order terms) the Dirac operator  
$D^+$ on the compact manifold $Z \cup (-Z)$ hence $\ind \bar D_1 = 0$.  
On the other hand, $\bar D_2 = D^+(\tilde X,g)  + \btilde - \delta\,
f^*(d\theta)$ hence $\ind \bar D_2$ is equal to the index of  
\[ 
D^+(\tilde X,g,\beta): L^2_{1,\delta}\,(\tilde X,S^+) \to L^2_{\delta}\,  
(\tilde X,S^-). 
\] 
The latter index is zero which can be seen as follows. Apply the  
Fourier--Laplace transform to the equation $D^+ (\tilde X,g,\beta)(u) = 0$ 
to obtain $D^+_{\mu} (X,g,\beta)\allowbreak (\hat u_{\mu}) = 0$. Since all 
$D^+_{\mu}(X,g,\beta)$ with  $\Re \mu = \delta$ are isomorphisms, we 
conclude that $\hat u_{\mu}  = 0$. Integrating over $I(\nu)$ with $\Re \nu 
= \delta$ gives $u = 0$, hence, $\ker D^+(\tilde X,g,\beta) = 0$. Similarly, 
$\coker D^+ (\tilde X,g,\beta) = 0$ by the same argument applied to the 
operator $D^- (\tilde X,g,\beta)$. 
\end{proof} 
 
 
\section{The change of index formula} \label{S:index-change}
This section is devoted to comparing the indices $\ind_{\delta} D^+(Z_+,g,
\beta)$ for different values of $\delta \in \R$. We continue to assume 
that the pair $(g,\beta)$ belongs to a regular path even though it need 
not be regular itself. The resulting formula \eqref{E:index} contains as 
a special case the formula of \cite[Theorem 1.2]{LM}\;(so in particular 
we provide a new and rather different proof of the latter). The indices 
$\ind D^+ (Z_+,g,\beta)$ for different $g$ and $\beta$ will be compared 
in the next section. 
 
 
\subsection{Reduction to an index problem on $\tilde X$} 
Given $\delta_1$, $\delta_2 \in \R$, consider a smooth function $\delta:  
\tilde X \to [0,1]$ such that $\delta (x) = \delta_1$ on $W_n$ with $n  
\le -1$ and $\delta (x) = \delta_2$ on $W_n$ with $n \ge 1$. Let $h(x) = 
\delta (x) \cdot \tilde f(x)$, and say that $\phi \in L^2_{k;\,\delta_1, 
\delta_2}\,(\tilde X,S^{\pm})$ if and only if $e^h\,\phi \in L^2_k\, 
(\tilde X,S^{\pm})$. In particular, if $\delta_1 = \delta_2 = \delta$,  
we get back the spaces $L^2_{k,\delta}(\tilde X,S^{\pm})$.  
As before, we have the commutative diagram  
\[ 
\begin{CD} 
L^2_{1;\,\delta_1,\delta_2}\,(\tilde X,S^+) @> D^+ >> L^2_{\delta_1, 
\delta_2}\,(\tilde X,S^-) \\ 
@VV e^h V @VV e^h V \\ 
L^2_1 (\tilde X,S^+) @> D^+ - dh >> L^2\,(\tilde X,S^-) 
\end{CD} 
\] 
whose vertical arrows are isomorphisms, and we conclude that  
\begin{equation}\label{E:dd} 
D^+ = D^+ (\tilde X,g,\beta):\; L^2_{1;\,\delta_1,\delta_2}\,(\tilde X,S^+) 
\to L^2_{\delta_1,\delta_2}\,(\tilde X,S^-) 
\end{equation}
is a Fredholm operator for all $\delta_1$, $\delta_2 \in \R$ away from  
a discrete set. We will denote its index by $\ind_{\delta_1,\delta_2}  
D^+(\tilde X,g,\beta)$.  
 
\begin{proposition} 
For any $\delta_1$, $\delta_2 \in \R$ away from a certain discrete set 
with no accumulation points,  
\[ 
\ind_{\delta_2} D^+(Z_+,g,\beta) - \ind_{\delta_1} D^+(Z_+,g,\beta) = 
\ind_{\delta_1,\delta_2} D^+(\tilde X,g,\beta). 
\] 
\end{proposition} 
 
\begin{proof} 
This is seen by a repeated application of the excision principle, the 
functions $h_1$, $h_2: Z_+ \to \R$ extending $\delta_1\cdot \tilde f(x)$ 
and $\delta_2\cdot\tilde f(x)$, respectively\,:
\begin{alignat*}{1} 
\ind_{\delta_2} D^+ (Z_+, & g, \beta) - \ind_{\delta_1} D^+(Z_+,g,\beta) \\ 
&= \ind (D^+(Z_+,g,\beta) - dh_2) - \ind (D^+(Z_+,g,\beta) - dh_1) \\  
&= \ind (D^+(Z_+,g,\beta) - dh_2) - \ind (D^+(Z^*_+,g,\beta) + dh_1) \\  
&= \ind (D^+(Z_+,g,\beta) - dh_2) + \ind (D^+(-Z^*_+,g,\beta) - dh_1) \\ 
&= \empty\hfil \ind (D^+(\tilde X,g,\beta) - dh)\;\;\, =\;  
\ind_{\delta_1,\delta_2} D^+(\tilde X,g,\beta). 
\end{alignat*} 
\end{proof} 
 
 
\subsection{Change of index via residues}\label{S:residues}
Our next goal will be to compute the index  $\ind_{\delta_1,\delta_2}  
D^+(\tilde X,g,\beta)$ in terms of the holomorphic family $D^+_{\mu} 
(X,g,\beta) = D^+ (X,g,\beta) - \mu\cdot f^*(d\theta)$.  
 
Fix a smooth function $\zeta: \tilde X \to \R$ such that $0\le \zeta  
\le 1$, $\zeta = 0$ on $W_n$ with $n \le -1$, and $\zeta = 1$ on $W_n$  
with $n \ge 1$. Let $u \in L^2_{1;\,\delta_1,\delta_2} (\tilde X,S^+)$  
be a solution of the equation $D^+ (\tilde X,g,\beta)(u) = 0$.  Write  
\[ 
u = (1 - \zeta)\,u + \zeta\,u = v + w, 
\] 
where $v = (1 - \zeta)\,u \in L^2_{1,\delta_1}(\tilde X,S^+)$ and $w =  
\zeta\,u\in L^2_{1,\delta_2}(\tilde X,S^+)$. A straightforward calculation  
shows that  
\[ 
D^+ (\tilde X,g,\beta) (v) = - k 
\quad\text{and}\quad 
D^+ (\tilde X,g,\beta) (w) =   k,
\] 
where $k = d\zeta\cdot u$. 

Since $k$ is supported in $W_0$, its Fourier--Laplace transform 
$\hat  k_{\mu}$ is obviously holomorphic as a function of $\mu$ in the 
entire complex plane. Apply Lemma~\ref{L:half-plane} to $w \in 
L^2_{1,\delta_2} (\tilde X,S^+)$ supported in $\tilde X_+$ to conclude 
that $\hat w_{\mu}$ is holomorphic in the half plane $\Re \mu < 
\delta_2$. A similar argument shows that $\hat v_{\mu}$ is holomorphic 
in the half plane $\Re \mu > \delta_1$. Hence the application of the 
Fourier--Laplace transform to the above two equations yields equations  
\[ 
D^+_{\mu}(X,g,\beta)\,(\hat v_{\mu}) = - \hat  k_{\mu} 
\quad\text{and}\quad 
D^+_{\mu}(X,g,\beta)\,(\hat w_{\mu}) =   \hat  k_{\mu},
\] 
that hold in the half planes $\Re \mu > \delta_1$ and $\Re \mu < 
\delta_2$, respectively.
 
The inverse $R_{\mu}$ of the holomorphic family $D^+_{\mu}(X,g,\beta)$ 
is a meromorphic function of $\mu$ in the entire complex plane; see 
Theorem~\ref{T:mero}. Therefore, away from the poles of $R_{\mu}$, we 
have the equations 
\[ 
\hat v_{\mu} = - R_{\mu}\, \hat  k_{\mu} 
\quad\text{and}\quad 
\hat w_{\mu} =   R_{\mu}\, \hat  k_{\mu}. 
\] 
This allows us to extend $\hat v_{\mu}$ and $\hat w_{\mu}$ to meromorphic  
functions in the entire complex plane, called again $\hat v_{\mu}$ and  
$\hat w_{\mu}$. Since (by Proposition~\ref{P:reed1}) the restriction of $\hat v_{\mu}$ to every  
interval $I(\delta_1 + i\alpha)$ is square integrable, we conclude that  
$\hat v_{\mu}$ does not have poles on $\Re \mu = \delta_1$. Similarly,  
$\hat w_{\mu}$ does not have poles on $\Re \mu = \delta_2$. 
 
The function $u = v + w$ with $v \in L^2_{1,\delta_1}\,(\tilde X,S^+)$  
and $w \in L^2_{1,\delta_2}\,(\tilde X,S^+)$ can now be recovered using  
the inverse Fourier--Laplace transform (see \eqref{E:zinv})\,: 
\begin{multline}\notag 
u(x + n)\; =\;  
\frac 1 {2\pi i}\,\int_{I(\delta_2 + i\alpha)} e^{-\mu (f(x)+n)}  
R_{\mu}\,\hat  k_{\mu}(x)\,d\mu \\ - 
\frac 1 {2\pi i}\,\int_{I(\delta_1 + i\alpha)} e^{-\mu (f(x)+n)}  
R_{\mu}\,\hat  k_{\mu}(x)\,d\mu. 
\end{multline} 
Let $\alpha$ be any real number such that $R_{\mu}\,\hat  k_{\mu}(x)$  
does not have poles on the horizontal lines $\Im\mu = \alpha$ and  
(consequently) $\Im\mu = \alpha + 2\pi$, and integrate $e^{-\mu(f(x)+n)} 
R_{\mu}\,\hat  k_{\mu}(x)$ over the positively oriented contour $\Gamma$  
shown in Figure~\ref{fig:contour}. 

\medskip

\begin{figure}[!ht] 
\centering 
\psfrag{A}{$\alpha$} 
\psfrag{B}{$\alpha + 2\pi$} 
\psfrag{0}{$0$} 
\psfrag{d}{$\delta_1$} 
\psfrag{w}{$\delta_2$} 
\psfrag{I0}{$I(\delta_1 + i\alpha)$} 
\psfrag{I1}{$I(\delta_2 + i\alpha)$} 
\psfrag{J1}{$J(\alpha)$} 
\psfrag{J2}{$J(\alpha + 2\pi)$} 
\includegraphics{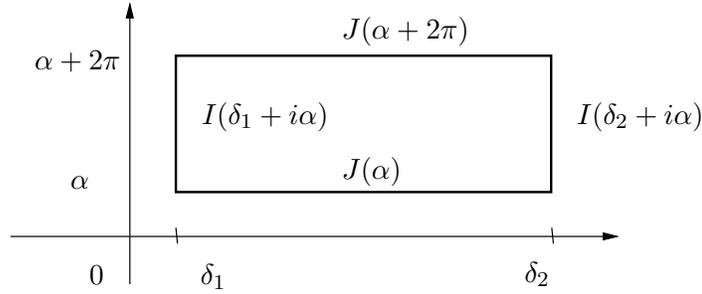} 
\caption{Contour $\Gamma$}
\label{fig:contour} 
\end{figure} 
  
\begin{lemma}\label{L:contour} 
The contributions to the above contour integral from the sides $J(\alpha)$  
and $J(\alpha + 2\pi)$ of $\Gamma$ cancel each other. 
\end{lemma} 
 
\begin{proof} 
Let $\mu \in J(\alpha)$ then $\mu + 2\pi i \in J(\alpha + 2\pi)$ and 
$D^+_{\mu + 2\pi i}(X,g,\beta) = D^+(X,g,\beta) - (\mu + 
2\pi i) f^*(d\theta) = e^{2\pi i f(x)} (D^+(X,g,\beta) - \mu f^* 
(d\theta)) e^{-2\pi i f(x)}$. Therefore,   
 \[ 
R_{\mu + 2\pi i} = e^{2\pi i f(x)}\,R_{\mu}\;e^{-2\pi i f(x)}. 
\] 
Together with the fact that $\hat  k_{\mu + 2\pi i}(x) = e^{2\pi i f(x)}  
\,\hat  k_{\mu}$ (see~\eqref{E:2pi}) this implies that 
 \[ 
e^{-(\mu + 2\pi i)(f(x)+n)} R_{\mu + 2\pi i}\;\hat  k_{\mu + 2\pi i}(x) 
= e^{-\mu (f(x)+n)} R_{\mu}\;\hat  k_{\mu}(x),  
\]  
and the statement follows.  
\end{proof} 
 
A straightforward application of the Cauchy integral formula then leads us  
to the formula   
\begin{multline}\notag 
u(x + n)\; =\; 
\frac 1 {2\pi i}\;\oint_{\Gamma}\; e^{-\mu (f(x)+n)} R_{\mu}\, 
\hat  k_{\mu}(x)\,d\mu \\ 
= \sum_j\; \Res_{\,\mu_j}\,\left(e^{-\mu (f(x) + n)}\, 
R_{\mu}\;\hat  k_{\mu}(x)\right), 
\end{multline} 
where $\mu_j$ are the poles of $R_{\mu}$ inside the contour $\Gamma$.  This 
formula describes the kernel of the operator \eqref{E:dd}. The residues will be 
explicitly calculated in the next section; note however that if $\delta_1 
\le \delta_2$ then both functions $\hat v_{\mu}$ and $\hat w_{\mu}$ are 
holomorphic inside $\Gamma$ hence the kernel of the operator \eqref{E:dd} 
vanishes. 
 
Next, the cokernel of $D^+(\tilde X,g,\beta)$ in the $L^2_{\delta_1,\delta_2}$ 
norm on $\tilde X$ is isomorphic to the kernel of $D^- (\tilde X,g,\beta)$ in  
the $L^2_{-\delta_1,-\delta_2}$ norm, and hence can be calculated in terms  
of the residues as above. Again, this cokernel vanishes whenever  
$-\delta_1 \le -\delta_2$, that is, $\delta_2\le \delta_1$. Therefore,  
depending on which of the weights $\delta_1$ or $\delta_2$ is larger,  
either the kernel or the cokernel of the operator \eqref{E:dd} vanishes. 
We will assume without loss of generality that $\delta_2 \le \delta_1$, 
so that $\ind_{\delta_1,\delta_2} D^+(\tilde X,g,\beta)$ is equal to the 
dimension of the kernel of $D^+(\tilde X,g,\beta)$ described by the above 
residue formula.  
 
 
\subsection{Calculating the residues} 
Let $\mu_j$ be a pole of $R_{\mu}$, and write the Laurent series of $R_{\mu}\, 
\hat  k_{\mu}(x)$ near $\mu_j$ in the form 
\[ 
R_{\mu}\,\hat  k_{\mu}(x) = \sum_{\ell = -m}^{\infty}\; b_\ell (x)\,
(\mu-\mu_j)^\ell,
\] 
for some spinors $b_{\ell}(x)$. Apply the operator $D_{\mu} = D^+_{\mu}
(X,g,\beta)$ to both sides of this equality to obtain 
\begin{alignat*}{1} 
\hat  k_{\mu} (x)  
&= \sum_\ell\; D_{\mu} b_\ell (x)\,(\mu - \mu_j)^\ell \\ 
&= \sum_\ell\; D_{\mu_j} b_\ell (x)\,(\mu - \mu_j)^\ell -    
\sum_\ell\; f^*(d\theta)\,b_\ell (x)\,(\mu - \mu_j)^{\ell+1}. 
\end{alignat*} 
The fact that $\hat  k_{\mu}$ is an entire function then implies that  
the coefficients $b_\ell (x)$ solve the system  
\begin{equation}\label{E:system} 
\begin{cases} 
\quad D_{\mu_j} b_{-1} = f^*(d\theta)\,b_{-2},\\  
\qquad \cdots \\ 
\quad D_{\mu_j} b_{-m+1} = f^*(d\theta)\,b_{-m},\\  
\quad D_{\mu_j} b_{-m} = 0. 
\end{cases} 
\end{equation}  

\medskip\noindent
The spinors $b_{-1}, \ldots, b_{-m}$ determine the residues as follows. Write  
\begin{multline}\notag 
e^{-\mu (f(x) + n)}\,R_{\mu}\;\hat  k_{\mu}(x)  
= e^{-\mu_j (f(x) + n)}\,e^{-(\mu-\mu_j) (f(x) + n)}\,R_{\mu}\; 
\hat  k_{\mu}(x) \\ 
= e^{-\mu_j (f(x) + n)}\,\sum_{k = 0}^{\infty}\;\frac {(-1)^k} 
{k\,!}\,(f(x) + n)^k\,(\mu-\mu_j)^k\;  
\sum_{\ell = -m}^{\infty}\; b_\ell (x)\,(\mu - \mu_j)^\ell, 
\end{multline} 
so that the residue of $e^{-\mu (f(x) + n)}\,R_{\mu}\;\hat  k_{\mu}(x)$ at 
$\mu_j$ equals 
\[ 
e^{-\mu_j(f(x) + n)}\,\sum_{p=1}^m\;(-1)^{p-1}\,(f(x)+n)^{p-1}\,  
b_{-p}(x)/(p-1)! 
\]  
Keeping in mind that $\tilde f(x) + n = \tilde f (x + n)$ for all $x \in 
W_0$ and  all integers $n$, we can re-write the latter formula as 
\begin{equation}\label{E:one} 
e^{-\mu_j \tilde f(x)}\,\sum_{p=1}^m\;(-1)^{p-1}\,\tilde f(x)^{p-1}\,  
b_{-p}(x)/(p-1)! 
\end{equation}  

Denote by $d(\mu_j)$ the number of linearly independent solutions of 
the equation $D^+(\tilde X,g,\beta)(u) = 0$ of the 
form~\eqref{E:one}. Equivalently, $d(\mu_j)$ is the dimension of the 
vector space of solutions of system~\eqref{E:system}. Note that 
$\ker D_{\mu_j} = 0$ if and only if $d(\mu_j) = 0$.  

\begin{remark}
Expand the meromorphic function $R_{\mu}$ near its pole $\mu_j$ into Laurent 
series,
\[
R_{\mu}\; = \sum_{\ell = -m}^{\infty}\; A_{\ell}\, (\mu - \mu_j)^{\ell}.
\]
The equation $D_{\mu} R_{\mu} = I$ then implies that the operators $A_{\ell}$
solve the system
\[
\begin{cases}
\quad D_{\mu_j} A_{-1} = f^*(d\theta)\,A_{-2},\\  
\qquad \cdots \\ 
\quad D_{\mu_j} A_{-m+1} = f^*(d\theta)\,A_{-m},\\  
\quad D_{\mu_j} A_{-m} = 0 
\end{cases} 
\]

\smallskip\noindent
similar to \eqref{E:system}. In particular, if $0 \ne b \in \im A_{-1}$ 
so that $b = A_{-1} (a)$ then setting $b_{-j} = A_{-j} (a)$ gives a 
solution of the system \eqref{E:system}, with $b_{-1} = b$. Since $A_{-1}$ 
is in fact the operator $P_{\mu_j}$ defined in \eqref{E:proj}, we readily 
conclude that
\begin{equation}\label{E:rank}
\rk P_{\mu_j}\; \le\; d(\mu_j).
\end{equation}
\end{remark}
 
 
\subsection{Change of index formula}\label{SS:index-change}
The proof of Lemma~\ref{L:contour} tells us that the operators 
$D^+_{\mu + 2\pi i}\,(X,g,\beta)$ and $D^+_{\mu}(X,g,\beta)$ are isomorphic.
We will use this fact to write
\[
D^+_z (X,g,\beta) = D^+(X,g,\beta) - \ln z\cdot f^*(d\theta)
\]
for $D^+_{\mu} (X,g,\beta)$ (compare with \eqref{E:Dz}) and also $d(z)$ 
for $d(\mu)$ and $P_z$ for $P_{\mu}$ if $z = e^{\,\mu}$. For any $\delta 
\le \delta'$ that make the operator \eqref{E:taubes} Fredholm, we have 
the following change of index formula\,:
\begin{equation}\label{E:index} 
\ind_{\delta} D^+(Z_+,g,\beta) - \ind_{\delta'} D^+(Z_+,g,\beta) =  
\sum_{e^{\delta} < |z| < e^{\delta'}} \; d (z). 
\end{equation} 
 
 
\section{The spectral flow formula}\label{S:metric} 
In this section, we will describe how the index of the operator \eqref{E:dirac}
changes along a special path $(g_I,\beta_I)$. The argument is strongly 
intertwined with the discussion of parameterized Seiberg--Witten moduli 
spaces $\M_I$ in Section~\ref{S:moduli}. The change of index formula (see 
Theorem \ref{T:sf}), which we refer to as the \emph{spectral flow formula}, 
is much more precise than the formula~\eqref{E:index} of the previous 
section.


\subsection{The reducibles}\label{S:MI}
We begin by reviewing the transversality of the intersection $\M^{0}_I = 
\widetilde \M_I\,\cap\,\p\mZ_I$ for special paths $(g_I,\beta_I)$. 
Understanding this transversality in very concrete terms will be crucial 
for our discussion.

Let $(g_I,\beta_I)$ be a special path as in Theorem~\ref{T:special}
that makes $\M_I$ regular and, in particular, $\M^{0}_I$ at most finite. 
Suppose $\tau \in I$ is such that $\M^{0}(X,g_{\tau},\beta_{\tau})$ is 
not empty. After a change of coordinates on $I$, we may assume that 
$\tau = 0$ and that, for all $t$ sufficiently close to zero, $g_t$ is 
constant. We will use the notations $g_0 = g$ and $\beta_0 = \beta$.

Let $[0,A,0,\phi]$ be a point in $\M^{0}(X,g,\beta) \subset \M^{0}_I = 
(\p\chi_I)^{-1} (d^+\beta_I)$. The fact that $\p\chi_I$ is transversal to 
the section $d^+\beta_I$ means that the linearization of the map $\p\chi_I 
- d^+\beta_I: \p\mZ_I \to \Omega^2_+ (X,i\R)$ at $[0,A,0,\phi]$ is a 
surjective linear operator. This operator will be called $\D$. The kernel 
of $\D$, which is necessarily zero dimensional, is then the tangent space 
to $\M^{0}_I$ at $[0,A,0,\phi]$. 

Using special paths allows us to avoid differentiating the metric, which 
provides for a particularly simple formula for $\D$. More precisely, the 
operator $\D$ in question is the operator
\[
\D: \R \times \Omega^1 (X;i\R) \times \Gamma(S^+)^{\perp}\longrightarrow 
\H^0(X,i\R)^{\perp} \times \Omega^2_+ (X,i\R) \times \Gamma(S^-)
\]
given by 
\[
\D (v,b,\psi) = (-d^*b,\; d^+b - v\;d^+ \dot\beta,\; D^+_A (X,g)\,(\psi) + 
b \cdot \phi),
\]
compare with~\cite[Lemma 27.1.1]{KM}. Here, $\Gamma(S^+)^{\perp}$ 
consists of all spinors $\psi$ such that $<\phi,\psi>_{L^2}\,= 0$, and 
$\H^0 (X;i\R)^{\perp} \subset \Omega^0 (X,i\R)$ consists of all functions 
$h: X \to i\R$ perpendicular to the subspace of constant functions $\H^0 
(X;i\R) = i\R$. The notation $\dot\beta$ means the derivative of 
$\beta_t$ with respect to $t$ evaluated at $t = 0$. 

Let us change variables $(v,b,\psi)$ to $(v,a,\psi)$ with $a = b - v \dot
\beta$. The above operator then takes the form
\[
\D (v,a,\psi) = (-d^*a,\;d^+a,\;D^+_A (X,g)\,(\psi) + a\cdot\phi + v \dot 
\beta \cdot \phi)
\]
(remember that $d^* \beta_t = 0$ for all $t$, hence $d^*\dot\beta = 0$).
In plain terms, the vanishing of $\ker \D$ means that the following system 
of equations on $v$, $a$, and $\psi$, has a unique solution $(v,a,\psi) = 
(0,0,0)$:

\begin{equation}\label{E:sys1}
\begin{cases}
\; d^* a = 0, & \ \\
\; d^+ a = 0, & \ \\
D^+_A (X,g)\,(\psi) + a\cdot\phi + v\dot\beta\cdot\phi = 0, & \ \\
<\phi,\psi>_{L^2}\,= 0. \ &
\end{cases}
\end{equation}

\medskip


\subsection{Harmonic functions}
Before we go on to deduce our spectral flow formula, we need to fix the
function $f: X \to \R$ that was built into the definition of $\ind D^+
(Z_+,g,\beta)$ but has remained pretty much arbitrary until now. We will
choose $f$ to be harmonic. The following existence and uniqueness result 
for harmonic functions can be found in Eells and Lemaire~\cite[Section 7]{EL}.

\begin{lemma}\label{L:harmonic}
For any metric $g$ on $X$, there exists a function $f: X \to S^1$ that 
is harmonic with respect to $g$ and has the property that $f^*(d\theta)$ 
represents the generator $1 \in \Z = H^1 (X;\Z)$. Moreover, such an $f$ 
is unique up to translation of $S^1$. 
\end{lemma}


\subsection{Change of index as spectral flow}\label{S:change}
Let $(g_0,\beta_0)$ and $(g_1,\beta_1)$ be two regular pairs of metrics and 
perturbations connected by a special path $(g_I,\beta_I)$ as in Theorem 
\ref{T:special} so that $\M_I$ is regular and $g_I$ is constant near each 
$t \in I$ where $\M^{0}(X,g_t,\beta_t)$ is non-empty. Choose a smooth path 
of functions $f_t: X \to S^1$ harmonic with respect to $g_t$. Denote by 
$d_t(z)$ the dimension of the space of solutions of system~\eqref{E:system} 
corresponding to the choice of metric $g_t$ and perturbation $\beta_t$. 

\begin{proposition}\label{P:gen}
The set of pairs $(t,z)$ with $\ker D^+_z (X,g_t,\beta_t) \ne 0$ and
$|z| = 1$ is in bijective correspondence with the points in $\M^{0}_I$. 
Moreover, for any $(t,z)$ in this set, $d_t (z) = \rk P_z = 1$.
\end{proposition} 

\begin{proof} 
Let $\tau \in I$ be such that $\M^{0}(X,g_{\tau},\beta_{\tau}) \subset 
\M^{0}_I$ is not empty. After changing coordinates, we will assume as in 
Section~\ref{S:MI} that $\tau = 0$ and write $g_0 = g$, $f_t = f$, $d_t 
(z) = d(z)$, and $\beta_0 = \beta$. A quadruple $[0,A,0,\phi]$ belongs 
to $\M^{0}(X,g,\beta)$ if and only if $F^+_A = d^+ \beta$, $D^+_A (X,g)
\,(\phi) = 0$, and $\|\phi\|_{L^2} = 1$. Up to gauge equivalence, $A = 
\beta - \ln z\cdot f^*(d\theta)$ for some $z \in \C$ with $|z| = 1$. 
Therefore, $D^+_z (X,g,\beta) = D^+_A (X,g)$, and we have the claimed 
bijective correspondence. 

We know from Proposition~\ref{P:kernels} that $\dim_{\,\C}\ker D^+_z (X,g,
\beta) = 1$, so that $d (z) \ge 1$. Suppose that $d (z) > 1$. Then the 
last two equations of the system~\eqref{E:system}, 
\[
\begin{cases}
\; D^+_z (X,g,\beta)\,(b_{-m+1}) = f^*(d\theta)\cdot b_{-m} \\
\; D^+_z (X,g,\beta)\,(b_{-m}) = 0
\end{cases}
\]
have a solution with $b_{-m+1} \ne 0$ and $b_{-m} \ne 0$. Without loss of 
generality, we may assume that $\|b_{-m}\|_{L^2} = 1$ and 
$<b_{-m},b_{-m+1}>_{L^2}\, = 0$. But then the triple $(v,a,\psi) = 
(0,- if^*(d\theta),\,ib_{-m+1})$ is a non-zero solution of the 
system~\eqref{E:sys1} with $\phi = b_{-m}$ and $A = \beta - \ln z\cdot 
f^*(d\theta)$, a contradiction. That $\rk P_z = 1$ now follows from 
\eqref{E:rank}.
\end{proof}
 
\medskip 
 
\begin{figure}[!ht] 
\centering 
\psfrag{C}{$\C$} 
\psfrag{t}{$t$} 
\psfrag{G}{$\Sm$} 
\includegraphics{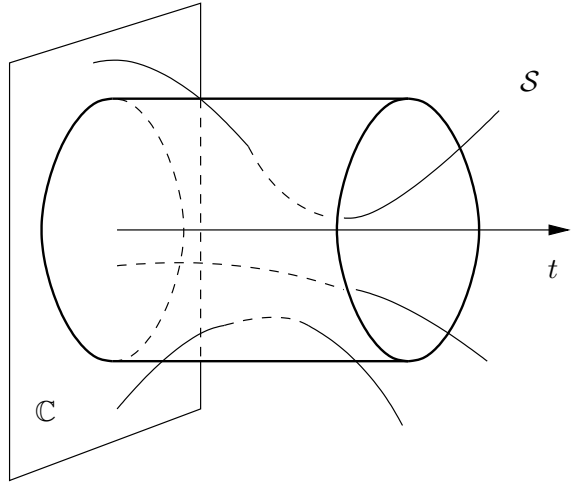} 
\label{fig:sf} 
\caption{Spectral curves}
\end{figure} 
 
Let $D^+_z (X,g_t,\beta_t)$ be a family of perturbed Dirac operators as 
above and consider the subset
\[ 
\Sm_I = \{\,(t,z)\in [0,1] \times \C^*\;|\;\ker D^+_z\,(X,g_t,\beta_t)
\ne 0\,\} 
\] 
of $[0,1] \times \C^*$. Its projection onto the second coordinate is in
essence the parameterized spectral set $\Ss_I$ defined in \eqref{E:Ss}. 
According to Theorem~\ref{T:spectra} and Proposition~\ref{P:gen}, the set 
$\Sm_I$ is just a finite family of smooth curves near the cylinder 
$C =\{\,(t,z)\in [0,1]\times \C^*\;|\;|z| = 1\,\}$. They will be referred 
to as the \emph{spectral curves}. These curves intersect the cylinder $C$ 
in finitely many points $(t,z)$ with multiplicity one. We will show later 
in Lemma \ref{L:non-sing} that these intersections are in fact transversal. 

The \emph{spectral flow} $\SF (D^+_z(X,g_I,\beta_I))$ along a special path 
$(g_I,\beta_I)$ is defined by counting the intersection points 
$\Sm_I\,\cap\,C$ with  sign $-1$ if the spectral curve is entering the 
cylinder $C$, and $+1$ if it is leaving.
 
\begin{theorem}\label{T:sf}
Let $\beta_0$, $\beta_1 \in \P$ be such that the pairs
$(g_0,\beta_0)$ and $(g_1,\beta_1)$ are regular. Then the operators 
\[
D^+(Z_+,g_0,\beta_0),\, D^+(Z_+,g_1,\beta_1):\;
L^2_1 (Z_+,S^+) \to L^2 (Z_+,S^-)
\]
are Fredholm on any periodic end manifold $Z_+$ whose end is modeled on 
$\tilde X$, and their indices differ by $\SF (D^+_z (X,g_I,\beta_I))$ 
for any special path of metrics and perturbations.  
\end{theorem} 
 
\begin{proof} 
Since $\M^{0} (X,g_0,\beta_0)$ and $\M^{0} (X,g_1,\beta_1)$ are both empty, 
the families $D^+_z (X,g_0,\beta_0)$ and $D^+_z (X,g_1,\beta_1)$ have zero 
kernels on the unit circle $|z| = 1$ by Proposition~\ref{P:gen}. The 
operators $D^+(Z_+,g_0,\beta_0)$ and $D^+(Z_+,g_1,\beta_1)$ are then 
Fredholm by Proposition~\ref{P:taubes}. 

Suppose that $\tau \in (0,1)$ is such that not all operators 
$D^+_z (X,g_{\tau}, \beta_{\tau})$ with $|z| = 1$ have zero kernel. After 
changing coordinates, we will assume that $\tau = 0$. Let $z_j$ ($j = 1,
\ldots,m$) be all the points in the complex plane such that $|z_j| = 1$ 
and $d_0 (z_j) = 1$.  
 
Choose $\ep_0 > 0$ so that the portions of all the spectral curves through 
$(0,z_j)$, $j = 1,\ldots,m$, cut out by the condition $0 < |t| \le \ep_0$, 
do not intersect the cylinder $C$ and have the property that $d_t(z) \le 1$. 
By continuity of spectral curves, for any small $\delta > 0$, one can find 
$\ep > 0$ such that $\ep < \ep_0$ and the intersection of $\Sm_I$ with the 
cylinder $\{\,(t,z)\;|\;|t| \le \ep,\, |z| = e^{\delta}\,\}$ is empty. Then 
we have well defined indices $\ind_{\delta} D^+ (Z_+,g_t,\beta_t)$ for all 
$t \in [- \ep,\ep]$ and, by continuity of the index,  
\[ 
\ind_{\delta} D^+ (Z_+,g_{-\ep},\beta_{-\ep}) = \ind_{\delta} D^+ 
(Z_+,g_{\ep},\beta_{\ep}). 
\] 
On the other hand, equation~\eqref{E:index} provides us with the formulas 
\[ 
(\ind - \ind_{\delta}) (D^+ (Z_+,g_{-\ep},\beta_{-\ep})) =  
\sum_{1 < |z| < e^{\delta}} \; d_{-\ep}(z)
\]
and
\[
(\ind - \ind_{\delta}) (D^+ (Z_+,g_{\ep},\beta_{\ep})) =  
\sum_{1 < |z| < e^{\delta}} \; d_{\ep}(z)
\]
where $d_{\pm \ep}(z)$ is zero, except at finitely many points where it 
is one. Combining the three formulas above, we obtain 
\begin{multline}\notag 
\ind D^+ (Z_+,g_{\ep}, \beta_{\ep}) -  
\ind D^+ (Z_+,g_{-\ep}, \beta_{-\ep}) \\ = 
\sum_{1 < |z| < e^{\delta}} \; d_{\ep}(z) - 
\sum_{1 < |z| < e^{\delta}} \; d_{-\ep}(z). 
\end{multline} 
One can further observe that, if $d_{-\ep}(z) = 1$ then $(-\ep,z)$ belongs 
to the same component of $\Sm_I$ as $(0,z_j)$ for some $j$; the latter 
contributes $-1$ to the spectral flow. Similarly, if $d_{\ep}(z) = 1$ then  
$(\ep,z)$ belongs to the same component of $\Sm_I$ at $(0,z_k)$ for some $k$; 
the latter contributes $+1$ to the spectral flow. 
\end{proof}


\section{The invariant}\label{S:invariant}
In this section, we will define the invariant $\lambda_{\,\SW}(X)$ and prove 
the first statement of Theorem \ref{T:main} regarding it.


\subsection{The invariant}
Let $g$ be an arbitrary metric on $X$ and choose a regular pair $(g,\beta)$
of metric and perturbation. Define 
\[
\lambda_{\,\SW} (X) = \#\,\M(X,g,\beta) - w\,(X,g,\beta).
\]

\begin{theorem}\label{T:comparison}
\; $\lambda_{\SW} (X)$ is independent of the choice of regular pair 
$(g,\beta)$.
\end{theorem}

\begin{proof}
Given two regular pairs, $(g_0,\beta_0)$ and $(g_1,\beta_1)$, choose a 
special path $(g_I,\beta_I)$ as in Theorem~\ref{T:special} so that the 
parameterized moduli space $\M_I$ is regular and the metric $g_t$ is 
constant near every value of $t \in I$ where $\M^{0}(X,g_t,\beta_t)$ is 
non-empty. According to Theorem~\ref{T:reg}, 
\[
\#\,\M (X,g_1,\beta_1) - \#\M (X,g_0,\beta_0)\, =\, \#\,\M^{0}_I,
\]
where $\#\,\M^{0}_I$ stands for the signed count of points in $\M^{0}_I$.
On the other hand, Theorem~\ref{T:sf} tells us that
\[
w\,(X,g_1,\beta_1) - w\,(X,g_0,\beta_0)\, =\, \SF (D^+_z (X,g_I,\beta_I)),
\]
a signed count of points on $\Sm_I\,\cap\,C$. According to Proposition 
\ref{P:gen}, the points in $\M^{0}_I$ and in $\Sm_I\,\cap\,C$ are in a 
bijective correspondence. That the corresponding points in $\M^{0}_I$ and
$\Sm_I\,\cap\,C$ are counted with the same sign is proved in the following
section.
\end{proof}


\subsection{Comparing signs}\label{S:signs}
We continue with the calculation that we star\-ted in Section~\ref{S:MI}.
To figure out the orientation of $[0,A,0,\phi] \in \M^{0}_I$, consider 
the path of Fredholm operators 
\[
\D_u (v,a,\psi) = (-d^*a,\; d^+a,\; D^+_A (X,g)\,(\psi) + ua\cdot\phi + 
uv\dot\beta\cdot\phi)
\]
parameterized by $u \in [0,1]$.  It connects the operator $\D_0 = (-d^*\,
\oplus\,d^+)\,\oplus\,D^+_A (X,g)$ to our operator $\D = \D_1$. We will 
compute the orientation of $[0,A,0,\phi]$ by calculating the orientation 
transport along $\D_u$ using the formula (1.5.9) from~\cite{Nic}. That 
formula expresses the orientation transport as the product 
\begin{equation}\label{E:ot}
\sign \det (R_0)\, \cdot\, \sign \det (R_1)\, \cdot \prod_{u \in [0,1)}\; 
(-1)^{\,\dim\ker \D_u},
\end{equation}
where $R_u: \ker \D_u \to \coker \D_u$, $u = 0, 1$, are \emph{resonance 
operators} defined as the derivative $(d/du)(\D_u - \D_0)$ evaluated 
at $u = 0, 1$, followed by the $L^2$ orthogonal projection $\pi$ onto 
$\coker \D_u$.

First we observe that, for any $u \ne 0$, the operator $\D_u$ has zero 
kernel, because any non-zero solution $(v,a,\psi)$ of the system
\begin{equation}\label{E:sys2}
\begin{cases}
\; d^* a = 0, & \ \\
\; d^+ a = 0, & \ \\
D^+_A (X,g)\,(\psi) + ua\cdot\phi + uv\dot\beta\cdot\phi = 0, & \ \\
<\phi,\psi>_{L^2}\,= 0, \ &
\end{cases}
\end{equation}
would give a non-zero solution $(v,a,\psi/u)$ of system~\eqref{E:sys1}.
Since $\ker \D_0$ is even dimensional, the orientation transport is
simply the sign of the determinant of the resonance operator $R_0: 
\ker \D_0 \to \coker\D_0$. Of course, both $\ker \D_0$ and $\coker \D_0$ 
need to be oriented. 

A straightforward calculation shows that $\ker\D_0 = \R \oplus \H^1(X;i\R)
\cong \R^2$ and $\coker\D_0 = \coker D^+_A (X,g) = \C$. They are canonically 
oriented, the former by the choice of homology orientation and the latter 
by the complex structure on $\coker D^+_A (X,g)$. 

Fix an isomorphism $\R^2 = \ker \D_0$ sending $(v,c) \in \R^2$ to $(v,
icf^*(d\theta),0)\allowbreak \in \ker \D_0$ (remember that $f: X \to 
\mathbb R$ was chosen to be harmonic). Since $(\D_u - \D_0)(v,a,\psi) = 
(0,\;0,\; ua \cdot \phi + uv \dot\beta \cdot \phi)$, we conclude that 
the resonance operator $R_0: \R^2 \to \coker D^+_A$ can be written as
\begin{equation}\label{E:res}
R_0(v,c) = \pi\,(ic f^* (d\theta)\cdot\phi + v\dot\beta\cdot\phi).
\end{equation}

\begin{lemma} 
\quad $\pi(i f^*(d\theta) \cdot \phi) \ne 0$.
\end{lemma}

\begin{proof}
Suppose that, on the contrary, $\pi (if^*(d\theta) \cdot \phi) = 0$.  Then
there is $\psi$ such that $i f^*(d\theta) \cdot \phi = D^+_A (X,g)\,(\psi)$. 
Since $D^+_A (X,g)\,(\phi) = 0$, we may assume without loss of generality 
that $<\phi,\psi>_{L^2}\,= 0$. But then $(v,a,\psi) = (0, -if^*(d\theta),
\psi)$ is a non-zero solution of the system~\eqref{E:sys1}, a contradiction.
\end{proof}

\begin{corollary}\label{C:basis}
The spinors $\pi(f^*(d\theta)\cdot \phi)$ and $\pi(i f^*(d\theta) \cdot \phi)
= i \pi \allowbreak (f^*(d\theta)\cdot \phi)$ form a positively oriented basis 
in $\coker D^+_A (X,g) = \C$.
\end{corollary}

Let us next study $\pi(\dot\beta \cdot \phi) \in \coker D^+_A (X,g)$. Up to 
gauge equivalence, $A = \beta - \ln z\cdot f^*(d\theta)$ for some $z\in \C$ 
with $|z| = 1$ so that $D^+_A(X,g) = D^+_z(X,g,\beta)$, and the unit spinor
$\phi$ spans $\ker D^+_z (X,g,\beta) = \C$. Let $(t,z_t)$ be the spectral 
curve through $(0,z)$; it is a smooth curve if $t$ stays sufficiently close 
to $0$. Write $\ln z_t = a_t + ic_t$ so that $\ln z = a_0 + ic_0 = ic$, and 
consider a path of unit spinors $\phi_t$ such that $\phi_0 = \phi$ and 
\[
\phi_t \in \ker (D^+ (X,g) - (a_t + ic_t)\,f^*(d\theta) + \beta_t).
\] 
Differentiate the equation 
\[
(D^+(X,g) - (a_t + ic_t)\,f^*(d\theta) + \beta_t)(\phi_t) = 0
\]
with respect to $t$ at $t = 0$. Since $\dot D^+(X,g) = 0$, we obtain
\[
D^+ (X,g)(\dot\phi) + (- (\dot a + i\dot c) f^*(d\theta)\cdot
\phi + \dot\beta\cdot\phi) + (-ic f^* (d\theta) + \beta)\cdot \dot\phi = 0,
\]
or, equivalently, 
\begin{equation}\label{E:diff}
D^+_A (X,g)(\dot\phi)\; =\; (\dot a + i\dot c) f^*(d\theta) \cdot\phi - 
\dot\beta\cdot\phi.  
\end{equation}
Projecting onto $\coker D^+_A (X,g)$ (along the image of $D^+_A (X,g)$), 
we obtain 
\begin{alignat*}{1}
\pi (\dot\beta\cdot\phi) 
&= \pi ((\dot a + i\dot c) f^*(d\theta)\cdot\phi) \\
&= \dot a\;\pi (f^*(d\theta)\cdot\phi)) + \dot c\;\pi(i f^*(d\theta)
\cdot\phi).
\end{alignat*}
Therefore, with respect to the basis of $\coker D^+_A (X,g)$ given by 
Corollary~\ref{C:basis}, the resonance operator $R_0$ has the matrix
\[
R_0\; =\; 
\begin{pmatrix}
\dot a & 0 \\
\dot c & 1 
\end{pmatrix}.
\]

\begin{lemma}\label{L:non-sing}
The operator $R_0$ is non-singular, that is, $\det R_0 = \dot a \ne 0$.
\end{lemma}

\begin{proof}
Suppose on the contrary that $\dot a = 0$. Let $\psi = \dot\phi + \alpha
\phi$ and choose $\alpha \in \C$ so that $<\phi,\psi>_{L^2}\, = 0$. 
Then~\eqref{E:diff} implies that $(v,a,\psi) = (1, - i\,\dot c f^*(d\theta),
\allowbreak \psi)$ is a non-zero solution of the system~\eqref{E:sys1}, 
a contradiction.
\end{proof}

A straightforward application of the orientation transport formula 
\eqref{E:ot} gives us the following result.

\begin{corollary}\label{C:or1}
The point $[0,A,0,\phi] \in \M^{0} (X,g,\beta)$ is oriented by $\sign\,
(\dot a)$. 
\end{corollary}

This concludes the proof of Theorem~\ref{T:comparison}, because 
$\sign(\dot a) = \pm 1$ is precisely the contribution of the 
point $(0,z) \in \Sm_I\,\cap\, C$ corresponding to $[0,A,0,\phi] \in 
\M^{0}_I$ to the spectral flow. 


\section{Relation with the Rohlin invariant}\label{S:rohlin}
Let $X$ be an oriented spin smooth homology $S^1\times S^3$ with a fixed 
homology orientation, that is, a generator $1 \in H^1 (X;\Z)$. Choose a 
connected 3-manifold $Y \subset X$ dual to this generator. Note that $Y$ 
is canonically oriented, and inherits a spin structure from $X$. The 
Rohlin invariant of $X$ is defined as $\rho (X) = \sign\,(Z)/8 \pmod 2$, 
where $Z$ is any smooth compact spin manifold with spin boundary $\p Z = 
Y$; see ~\cite{ruberman:ds,ruberman-saveliev:survey,scharlemann:phs}. We 
will show that 
\[
\lambda_{\,\SW}(X) = \#\,\M(X,g,\beta) - \ind_{\,\C} D^+ (Z_+,g,\beta)
- \sign\,(Z)/8
\]
reduces mod 2 to the Rohlin invariant by arguing that first, $\# \M
(X,g,\beta)$ is even because $\M(X,g,\beta)$ has quaternionic structure, 
and second, that $\ind_{\,\C} D^+ (Z_+,g,\beta)$ is even because $D^+ 
(Z_+,g,\beta)$ is quaternionic linear. Neither is actually true unless 
we take special care of choosing proper metrics and perturbations, as 
described below. 


\subsection{Generic metrics}
Let $Z_+ = Z\,\cup\,\tilde X_+$ be an end-periodic manifold whose end is 
modeled on the infinite cyclic cover of $X$. For any choice of metric 
$g$ on $X$, the Dirac operator 
$D^+ (Z_+,g): L^2_1\,(Z_+,S^+)\to L^2\,(Z_+,S^-)$ is quaternionic linear,
hence its index $\ind_{\,\C} D^+ (Z_+,g)$ is even -- assuming of 
course that $D^+(Z_+,g)$ is Fredholm. So far we know two ways of ensuring 
Fredholmness. We can use Corollary~\ref{C:taubes} to conclude that 
$D^+ (Z_+,g): L^2_{1,\delta} (Z_+,S^+) \to L^2_{\delta} (Z_+,S^-)$ is 
Fredholm for all but a discrete set of $\delta \in \R$, or we can combine 
Propositions~\ref{P:reg1} and~\ref{P:gen} to conclude that $D^+(X,g,\beta): 
L^2_1\,(Z_+,S^+) \to L^2\,(Z_+,S^-)$ is Fredholm for generic $\beta \in \P$.  
However, introducing either a weight $\delta \ne 0$ or a perturbation 
$\beta \ne 0$ ruins the quaternionic linearity of the Dirac operator. 
The paper~\cite{RS} of the second and third authors provides a better way 
of achieving Fredholmness by perturbing the metric alone and hence 
preserving the quaternionic linearity. 

\begin{theorem}
The operator $D^+ (Z_+,g): L^2_1\,(Z_+,S^+)\to L^2\,(Z_+,S^-)$ is Fredholm
for a generic choice of metric $g$ on $X$. 
\end{theorem}

Note that a choice of metric $g$ as in the above theorem only guarantees 
(via Proposition \ref{P:gen}) that $\M(X,g,0)$ has no reducibles but not 
that it is regular. A further perturbation $\beta \in \P$ may be needed 
to ensure its regularity. If that perturbation is small enough, it will 
not create any reducibles in $\M(X,g,\beta)$. Moreover, according to 
Theorem~\ref{T:sf}, we will have 
\[
\ind_{\,\C} D^+ (Z_+,g,\beta) = \ind_{\,\C} D^+ (Z_+,g),
\]
meaning that the index $\ind_{\,\C} D^+ (Z_+,g,\beta)$ will be even 
despite the fact that $D^+ (Z_+,g,\beta)$ may no longer be quaternionic 
linear. 

\begin{corollary}
The index $\ind_{\,\C} D^+ (Z_+,g,\beta)$ is even for a generic metric $g$ 
and a generic sufficiently small perturbation $\beta \in \P$. 
\end{corollary}


\subsection{$J$--action} 
The quaternionic structures on $S^{\pm}$ lead to a natural $\Z/4$--action 
on the triples $(A,s,\phi)$ given by $J(A,s,\phi) = (-A,s,j\phi)$.  It is
free because of the condition $\|\phi\|_{L^2}\,= 1$, and it descends to a 
free involution $J: \tmZ \to \tmZ$.   The following result is 
straightforward, once we observe that $\tau(j\phi) = - \tau (\phi)$ in 
\eqref{E:sw}.

\begin{lemma}
The map $\swz: \tmZ \to \Omega^2_+ (X,i\R)$ is equivariant with
respect to $J$ in that the following diagram commutes

\[
\begin{CD}
\tmZ @> \swz >> \Omega^2_+ (X,i\R) \\
@V J VV @VV -1 V \\
\tmZ @> \swz >> \Omega^2_+ (X,i\R)
\end{CD}
\]
\end{lemma}

\medskip
In particular, we see that $J$ does not act on $\M (X,g,\beta)$ unless 
$\beta = 0$. Therefore, if we want to show that $\#\,\M (X,g,\beta)$ is 
even for a generic metric, we will need more elaborate perturbations. 
In what follows, we adopt the approach of~\cite{MSz}. 

Let us view $\swz$ as a section of the trivial bundle $\E = \tmZ
\times \Omega^2_+ (X,i\R) \to \tmZ$ given by $\swz\,([A,s,\phi]) = 
([A,s,\phi], F^+_A - s^2\,\tau (\phi))$. It is equivariant, meaning that 
the following diagram commutes

\[
\begin{CD}
\tmZ @> \swz >> \E \\
@V J VV @VV \sigma V \\
\tmZ @> \swz >> \E
\end{CD}
\]

\medskip\noindent
Here, $\sigma$ acts as $J$ on the base $\tmZ$ and as $-1$ on each 
of the fibers of $\E$. Taking quotient by the free action of $J$ and 
$\sigma$, we obtain a section 
\[
\swz': \tmZ'\; \to\; \E'
\]
of the bundle $\E' = \tmZ\,\times_{\sigma}\Omega^2_+\,(X,i\R)$ over 
the Hilbert manifold $\tmZ' = \tmZ/J$. 

Adding a small generic section $\zeta': \tmZ'\to \E'$ makes $\swz'
+ \zeta'$ transversal to the zero section of $\E'$, and also makes its 
lift $\swz + \zeta: \tmZ \to \E$ an equivariant section transversal 
to the zero section of $\E$. In fact, we can choose $\zeta = d^+ \gamma$ 
for a map $\gamma: \tmZ \to \P$ which is equivariant in that $\gamma 
([-A,s,j\phi]) = - \gamma ([A,s,\phi])$.  The perturbations $\beta \in \P$ 
that we used before can then be viewed as constant maps $\beta: \tmZ 
\to \P$. The perturbed Seiberg--Witten moduli space 
\[
\M (X,g,\gamma) = (\swz\, +\, d^+ \gamma)^{-1}(0)
\] 
is a compact regular manifold of dimension zero acted upon freely by the 
involution $J$, cf.~\cite{MSz}. In particular, $\M (X,g,\gamma)$ consists
of an even number of points. 

\begin{proposition}
For a generic metric $g$ and a sufficiently small gene\-ric perturbations 
$\beta$ and $\gamma$ as above, the moduli spaces $\M (X,g,\beta)$ and 
$\M (X,g,\gamma)$ are in bijective correspondence. In particular, the
number $\#\,\M (X,g,\beta)$ is even.
\end{proposition}

\begin{proof}
In the Hilbert space of all perturbations $\tmZ \to \P$, the zero 
perturbation corresponds to the moduli space $\M (X,g,0)$ with empty 
$\M^{0}(X,g,0)$ by our choice of generic metric $g$. By continuity, there is 
a small ball in this space centered at 0 such that all the perturbations 
$\eta$ in it have the property that $\M^{0}(X,g,\eta)$ is empty. Choose
$\beta$ and $\gamma$ sufficiently small so that they belong to this ball,  
and connect $\beta$ to $\gamma$ by the path $\eta_t = (1 - t)\beta + t
\gamma$, $0 \le t \le 1$. Along this path, all $\M (X,g,\eta_t)$ have empty 
$\M^{0}(X,g,\eta_t)$ but are not necessarily regular, except at the endpoints.
Perturb this path a little rel its endpoints into $\eta'_t$ so that $\eta'_t$ 
stays inside the ball, while keeping all $\M^{0} (X,g,\eta'_t)$ empty and 
making $\M (X,g,\eta'_t)$ regular. The parameterized moduli space 
\[
\bigcup_{t \in [0,1]} \; \{t\}\times \M(X,g,\eta'_t)
\]
provides an oriented cobordism between $\M (X,g,\beta)$ and 
$\M (X,g,\gamma)$, and the result follows. 
\end{proof}


\section{Negative-definite manifolds with $b_1 = 1$}\label{S:neg-def}
The definition of the invariant $\lambda_{\,\SW}(X)$ extends to a more general 
topological setup, encountered~\cite{okonek-teleman:b+0,teleman:definite} in 
the study of non-K\"ahler complex surfaces, particularly those of type 
$\rm{VII}_0$. We will present two sets of hypotheses that yield a 
well-defined invariant; unfortunately, both are stronger than the minimal 
hypothesis that $H_1 (X) = \Z$ and $b^2_{+} (X) = 0$ one would hope for.

In the first situation we assume that $X$ is a smooth oriented closed 
$4$-manifold and that the following conditions hold:
\begin{equation}\label{E:Z-split}
\begin{split}
 &H_1(X;\Z) = \Z,\;  b^2_{+}(X) =0, \; b^2_{-}(X) = n,\; \mathrm{and}\\ 
 &H_3 (X)\ \textrm{is generated by an integral homology sphere}\ 
Y\subset X. 
\end{split}
\end{equation}

\noindent
In the second situation we assume that
\begin{equation}\label{E:bminus=1}
H_1(X;\Z) = \Z,\;  b^2_{+}(X) =0,\; \mathrm{and}\;\; b^2_{-}(X) = 1
\end{equation}
but no longer make any special hypothesis about the topology of a $3$-manifold 
$Y$ generating $H_3(X)$.

By Donaldson's theorem, the intersection form of $X$ is diagonalizable over 
the integers. Let $\spincs_X$ be a $\spinc$ structure with $c_1(\spincs_X)$
dual to the sum of the vectors in a diagonalizing basis, so that  
\begin{equation}\label{E:spinc}
c_1(\spincs_X)^2 = \sign(X) = -n.
\end{equation}
Note that $H^2(X;\Z)$ is torsion-free, hence specifying $c_1(\spincs_X)$ actually determines $\spincs_X$.  Any two bases diagonalizing $H_2(X;\Z)$ differ by permutation and change of signs of the basis elements, so there are in principle $2^{n-1}$ choices of \spinc structure to consider when defining our invariant. (The `charge conjugation' of the Seiberg-Witten equations~\cite[Section 6.8]{morgan:swbook} implies that the invariant is preserved when one reverses all of the signs of the basis elements.) If $X$ is an $n$-fold blowup of a homology $S^1 \times S^3$, then these all yield the same invariant up to sign, but we do not know if this holds in general.

For any choice of \spinc structure $\spincs_X$ as above and any choice of 
metric $g$ on $X$, the index of the \spinc Dirac operator $D^+(X,\spincs_X,g)$ 
is given by 
\[
\left(c_1(\spincs_X)^2 - \sign(X)\right)/8 = 0.
\]
This implies that the Seiberg--Witten moduli space corresponding to 
$\spincs_X$ has formal dimension $0$. An analogue of Proposition \ref{P:reg1} 
then shows that, for a generic perturbation $\beta \in \P$, the perturbed 
Seiberg--Witten moduli space $\M(X,\spincs_X,g,\beta)$ is a compact 
zero-dimensional manifold with no reducibles. Choose a homology orientation 
on $X$, that is, a generator $1\in H^1 (X;\Z)$, then $\M(X,\spincs_X,g,\beta)$ 
is canonically oriented, and we denote by $\#\,\M(X,\spincs_X,g,\beta) $ the 
signed count of its points. 

As before, this count can change as the metric and perturbation vary, so we 
need to define a suitable index-theoretic correction term $w(X,\spincs_X,g,
\beta)$. To do this, choose a $3$-manifold $Y\subset X$ representing $1 \in 
H_3 (X;\Z)$, and let $\spincs_Y$ be the restriction of $\spincs_X$ to $Y$.  
If $b^2_{-}(X) >1$, then we assume as in \eqref{E:Z-split} that $Y$ is an 
integral homology sphere. Also choose a smooth oriented compact \spinc 
manifold $Z$ with $\partial(Z,\spincs_Z) = 
(Y,\spincs_Y)$. Note that a simply-connected $Z$ with $\p Z = Y$ will have a 
\spinc structure extending $\spincs_Y$, so there are many choices for $Z$. 
As in Section \ref{S:epmflds}, form the end-periodic \spinc manifold $Z_+ = 
Z\,\cup\,\tilde X_+$ and extend a regular pair $(g,\beta)$ to an end-periodic 
pair $(g,\beta)$ on $Z_+$. Use the \spinc structure $\spincs_Z$ to lift 
$\spincs_X$ to an end-periodic \spinc structure on $Z_+$, called
$\spincs_{Z_+}$. As in Theorem \ref{T:taubes2}, the $L^2$--closure of the 
operator $D^+(Z_+,\spincs_{Z_+},g) + \beta$ will be Fredholm. Define the 
correction term to be
\[
w\,(X,\spincs_X,g,\beta)\, =\, \ind_{\,\C} (D^+ (Z_+,\spincs_{Z_+},g) + \beta)\, 
+ \frac 1 8\left(\sign (Z) - c_1(\spincs_Z)^2\right).
\]
The following proposition is an analogue of Proposition \ref{P:3.2}. 

\begin{proposition} \label{P:indep}
Given either \eqref{E:Z-split} or \eqref{E:bminus=1}, the correction term 
$w\,(X,\spincs_X,g,\beta)$ is independent of the choice of $Y \subset X$ 
and $Z$, and of the way $g$, $\beta$, and the \spinc structure $\spincs_X$ 
are extended over $Z_+$.  
\end{proposition}

Given Proposition~\ref{P:indep}, the argument that proves Theorem \ref{T:main} 
can be used essentially word for word to prove the following result.

\begin{theorem}
Let $X$ be a smooth oriented homology oriented $4$-manifold satisfying 
conditions~\eqref{E:Z-split} or~\eqref{E:bminus=1}, and $\spincs_X$ a \spinc
structure satisfying \eqref{E:spinc}. Then $\lambda_{\,\SW}(X,\spincs_X) = 
\#\,\M(X,\spincs_X,g,\beta) - w(X,\spincs_X,g,\beta)$ is independent of the 
choice of metric $g$ and generic perturbation $\beta$.
\end{theorem}

\begin{proof}[Proof of Proposition~\ref{P:indep} assuming \eqref{E:Z-split}]
Given two choices, $Z$ and $Z'$ with $\p Z = \p Z' = Y$, we use the excision 
principle and the index theorem to obtain
\begin{align*}\notag
\ind_{\,\C} D^+ (Z'_+,\spincs_{Z'_+}, & g,\beta) - 
\ind_{\,\C} D^+ (Z_+,\spincs_{Z_+},g,\beta)& \\ 
&= \ind_{\,\C} D^+ (- Z \cup Z',\spincs_Z \cup \spincs_{Z'})&\\ 
&= \frac 1 8\left(c_1(\spincs_Z \cup \spincs_{Z'})^2 - \sign (-Z \cup Z')\right) &\\
&=\frac 1 8\left( \sign (Z) -c_1(\spincs_Z)^2\right) -\frac 1 8 \left(\sign (Z') - c_1(\spincs_{Z'})^2\right),&
\end{align*} 
which proves that $w(X,\spincs_X,g,\beta)$ is independent of the choices of $Z$ and the extensions.

Let $Y$ and $Y'$ be two integral homology spheres carrying the generator of $H_3 (X;\Z)$. Suppose first that $Y'$ is disjoint from $Y$, and write $X = U\, \cup\, U'$ with $-Y\,\cup\, Y' = \partial U$.  Let $\spincs_U$ and $\spincs_{U'}$ denote the restriction of $\spincs_X$ to $U$ and $U'$. To use the excision principle as in the proof of Proposition~\ref{P:3.2}, we need to show that 
\begin{equation}\label{E:square}
c_1(\spincs_U)^2 = \sign(U)\quad\text{and}\quad c_1(\spincs_{U'})^2 = \sign(U').
\end{equation}
We know that $H^2(X) = H^2(U)\,\oplus\, H^2(U')$, and that the intersection form on each summand is diagonalizable.  Therefore, $-c_1(\spincs_U)^2 \ge b_2(U)$ and $-c_1(\spincs_{U'})^2 \ge b_2(U')$.  On the other hand, $ c_1(\spincs_X)^2 = c_1(\spincs_U)^2  + c_1(\spincs_{U'})^2 $, and \eqref{E:square} follows.  

If $Y$ and $Y'$ are not disjoint, choose lifts of $Y$ and $Y'$ to $\tilde X$. Translate $Y'$ by a sufficiently high prime power $p$ of the covering translation to make it disjoint from $Y$. Then $Y$ and $Y'$ cobound a submanifold $U$ of $\tilde X$ which embeds into the cyclic $p$-fold cover $\hat{X} \to X$. Since $p$ is prime, it follows from Smith theory (cf.~\cite{gilmer:thesis}) that 
$b_1(\hat{X}) = 1$. Both the signature and Euler characteristic multiply 
by $p$ when passing to a cyclic $p$-fold cover hence $\chi(\hat{X}) = pn$ 
and $\sign(\hat{X}) = -pn$ and, in particular, $\hat{X}$ is negative 
definite. The argument above (with $H^2$ replaced by $H^2/\mathrm{torsion}$) 
now shows that \eqref{E:square} holds.
\end{proof}

\begin{proof}[Proof of Proposition~\ref{P:indep} assuming \eqref{E:bminus=1}]
As in the previous proof, we start with disjoint manifolds $Y$ and $Y'$, separating $X$ into $U$ and $U'$. When $Y$ and $Y'$ have nontrivial homology, the relation between the intersection form of $X$ and those of $U$ and $U'$ is best understood via Novikov additivity. Following the proof of~\cite[Proposition 7.1]{atiyah-singer:III}, define $A \subset H^2(X;\Q)$ as the image of the composition
\[
H^2(U, \partial U;\Q) \to H^2(X,U';\Q)\to H^2(X;\Q),
\] 
and define $A' \subset H^2(X;\Q)$ similarly. The subspaces $A$ and $A'$ are mutual annihilators for the intersection form on $X$. Since $H^2(X;\Q) = \Q$, at least one of these subspaces must vanish; in particular, $A \cap A' = 0$. It then follows as in \cite{atiyah-singer:III} that 
\begin{multline}\notag
H^2(X;\Q) \ = \ \im\,[H^2(U,\partial U;\Q) \to H^2(U;\Q)] \\ \oplus\; 
\im\,[H^2(U',\partial U';\Q) \to H^2(U';\Q)],
\end{multline}
where exactly one of the summands (say the former) is non-zero. The projections onto the summands are given by restriction, hence $c_1(\spincs_U) \in \im\,[H^2(U,\partial U;\Q) \to H^2(U;\Q)]$, and similarly for $c_1(\spincs_{U'})$. The latter of course vanishes because it lives in the trivial vector space. As for the former, $c_1(\spincs_U)^2 \in \Q$ is well-defined and, as claimed,  
\[
c_1(\spincs_{U})^2 = c_1(\spincs_{X})^2 = -1 = \sign(U).
\]

If $Y$ and $Y'$ intersect, then the argument above which allows to separate them by passing to a covering space will not work, because the condition $b_2 =1$ need not hold in the covering space.   Instead, we appeal to the following principle (cf.~\cite{lickorish:knot-theory,rice:s-equivalence}):   There exist a sequence of connected submanifolds $Y = Y_0,\ldots,Y_n = Y'$ such that each $Y_i$ carries the generator of $H_3(X)$, and $Y_{i+1}$ is disjoint from $Y_i$.   Thus, the invariants defined by cutting along each $Y_i$ in turn are equal, and the independence is proved. 
\end{proof}

Note that the argument at the end of the second proof would not work in the situation when $b_2 >1$, because there is no guarantee that the $Y_i$ would be homology spheres. 

\section{Examples}\label{S:examples}
In this section,  we discuss the invariant $\lambda_{\,\SW}(X)$ for mapping 
tori $X$ of finite order diffeomorphisms $\tau: Y \to Y$. We succeed in 
calculating $\lambda_{\,\SW} (X)$ explicitly for Seifert fibered homology 
spheres $Y$ and a certain natural $\tau$ associated with $Y$ as a link of 
a singularity.


\subsection{Mapping tori}
Let $Y$ be a closed oriented 3-manifold and $\tau: Y \to Y$ a finite order 
orientation preserving diffeomorphism such that the mapping torus $X = 
([0,1]\times Y)\,/\,(0,x) \sim (1,\tau(x))$ is a homology $S^1 \times S^3$.
Suppose that $Y$ is equipped with a Riemannian metric $g$ such that $\tau$
is an isometry; then the product metric on $[0,1] \times Y$ gives rise to a 
natural metric on $X$ which we call again $g$.

\begin{proposition}
The infinite cyclic cover $\tilde X$ of $X$ is isometric to $\R \times Y$.
\end{proposition}

\begin{proof}
The mapping torus $X$ can be written as  $X =  (\mathbb R \times Y)\,/\,
(t,x) \sim (t + 1, \tau(x))$, where the map $(t,x) \to (t + 1, \tau(x))$
is an isometry of the product metric.
\end{proof}

We conclude that a periodic end modeled on the mapping torus of a finite 
order isometry is isometric to a product end. In particular, the Dirac 
operator $D^+ (Z_+,g)$ is isomorphic over the end to $d/dt + D$, where 
$D$ is the self-adjoint Dirac operator on $Y$. According to~\cite{Maier}, 
the operator $D$ is invertible for a generic metric $g$ on $Y$, and hence 
the $L^2$--closure of $D^+ (Z_+,g)$ is Fredholm. Therefore,
\begin{multline}\notag
w (X,g,0) = w (S^1\times Y,g,0) \\ = \ind_{\C} D^+ (Z_+,g) + \frac 1 8\,
\sign (Z) = - \left(\frac 1 2\,\eta_{\Dir} (Y) + \frac 1 8\,\eta_{\Sign} 
(Y)\right)
\end{multline}
by the Atiyah--Patodi--Singer theorem~\cite{aps:I}, where $\eta_{\Dir} (Y)$ 
and $\eta_{\Sign} (Y)$ refer to the $\eta$-invariants of, respectively,
the Dirac operator $D$ and the odd signature operator on $Y$. 

This takes care of the correction term $w(X,g,0)$. The other ingredient in 
calculating $\lambda_{\,\SW} (X)$ is the signed count of points in the moduli 
space $\M(X,g,0)$. Since $X$ admits a fixed point free action of $S^1$ (which
makes it into a circle bundle over the orbifold $Y/\tau$), we can employ 
techniques of~\cite{baldridge} to identify $\M(X,g,0)$, for a generic metric 
$g$, with $\M^{\tau}(Y,g)$, the equivariant Seiberg--Witten moduli space on 
the 3-manifold $Y$. 

Note that using the above approach to computing $\lambda_{\SW} (X)$ for 
mapping tori is in general problematic because of the equivariant
transversality required from the metric $g$. However, this approach 
works in at least a couple of instances. One is when $\tau$ is the 
identity so that $X = S^1 \times Y$ is the product, and the other is 
when $Y$ is a Seifert fibered homology spheres and $\tau$ is a certain 
involution associated with $Y$ as a link of a singularity. These two 
classes of examples will be studied in the rest of this section.


\subsection{The product case}
Let $X = S^1 \times Y$ have the product metric. It follows from Lim 
\cite{lim:sw3} that, for a generic metric $g$ on $Y$, we have 
\begin{equation}\label{E:km}
\lambda_{\,\SW}(S^1 \times Y)\, =\, - \lambda (Y), 
\end{equation}
where $\lambda (Y)$ is the Casson invariant \cite{akbulut-mccarthy}.
This supports the conjecture stated in the introduction because we know 
that $\lambda_{\,\FO} (S^1\times Y) = \lambda (Y)$; 
see~\cite{ruberman-saveliev:mappingtori}.

\begin{remark}
Our orientation convention is that $\lambda (\Sigma(2,3,5)) = -1$. Since
the metric $g$ on $\Sigma(2,3,5)$ has positive scalar curvature, $\M
(\Sigma(2,3,5),g)$ is empty. On the other hand, if we choose $Z$ to 
be the plumbed manifold with the (negative definite) intersection form 
$E_8$ then $\sign (Z)/8 = -1$ and $\ind D^+ (Z_+,g) = 0$ (the latter can
be found in~\cite[Proposition 8]{froyshov:mrl}). This fixes the sign in 
formula~\eqref{E:km}.
\end{remark}


\subsection{Seifert fibered homology spheres}
Given pairwise relatively prime integers $a_1,\ldots,a_n \ge 1$, consider
the Seifert fibered homology sphere $Y = \Sigma (a_1,\ldots,a_n)$. This 
is an integral homology sphere which will be viewed as a link of a 
Brieskorn--Hamm complete intersection singularity with real coefficients. 
It is canonically oriented and admits a fixed point free circle action. 
This action makes $Y$ into an orbifold circle bundle $\pi: Y \to F$, where 
$F$ is the 2--sphere with $n$ singular points of multiplicities $a_1,\ldots,
a_n$. The orbifold Euler characteristic of $F$ is given by the formula
\[
\chi (F) = 2 - \sum \; (1 - 1/a_k).
\]

Let $i\eta$ denote the connection form of the circle bundle and $g_F$ an 
orbifold metric on $F$ with constant curvature. We follow~\cite{MOY} and 
endow $Y$ with the metric $g = \eta^2 + \pi^*(g_F)$ and the connection 
$\mathring\empty\,\nabla$ on $TY$ canonically induced by the Levi--Civita 
connection on $F$. Note that $\mathring\empty\,\nabla$ differs from the 
standard Levi--Civita connection on $Y$ used in the definition of the 
Seiberg--Witten invariants. 

According to~\cite[Section 2.3]{Nic3}, the metric $g$ on $Y$ is generic 
in that $\ker D(Y,g) = 0$, and all irreducible solutions are 
$S^1$--invariant up to gauge transformation. This implies that the moduli 
space $\mathring\empty\M(Y,g)$ of irreducible solutions to the 
Seiberg--Witten equations on $Y$ with respect to the metric $g$ and the 
connection $\mathring\empty\,\nabla$ can be identified via pullback with 
two copies of the space of effective orbifold divisors over $F$ with 
orbifold degree not exceeding $-\chi(F)/2$. More precisely, $\mathring
\empty\M(Y,g)$ contains finitely many components $\mathcal C^+(\ep)$ 
and an equal number of components $\mathcal C^-(\ep)$, both labeled by 
the vectors $\ep = (\ep_1,\ldots,\ep_n)$ such that $0 \le \ep_k < a_k$ and
\[
\sum\; \ep_k/a_k \; \le \; - \chi (F) /2.
\]
The components $\mathcal C^{\pm}(\ep)$ consist of holomorphic, respectively, 
anti-holo\-morphic, vortices on $F$. If $n = 3$ or $n = 4$, each of the 
components $\mathcal C^{\pm}(\ep)$ is just a point.

Regarding the Seiberg--Witten moduli space $\M(Y,g)$ corresponding to the
metric $g$ and the Levi--Civita connection, Nicolaescu~\cite[Theorem 3.1]
{Nic1} showed that there is a natural bijective correspondence between 
$\mathring\empty\M(Y,g)$ and $\M(Y,g)$ for all metrics $g$ on $Y$ as 
above with sufficiently short circle fibers.  

Let $\tau: Y \to Y$ be induced by the complex conjugation on the link $Y$. 
Then $\tau$ is an involution that makes $Y$ into a double branched cover 
of $S^3$ with branch set a Montesinos knot. It commutes with the projection 
$\pi: Y \to F$ and thus defines an anti-holomorphic involution on $F$. In 
particular, it interchanges the holomorphic and anti-holomorphic vortices 
$\mathcal C^{\pm}(\ep)$ on $F$ so that $\mathring\empty\M^{\tau}(Y,g)$ and 
hence $\M^{\tau}(Y,g)$ are empty. Therefore, for the mapping torus $X$ of 
$\tau$,
\[
\lambda_{\,\SW} (X)\; =\; \frac 1 2\;\eta_{\Dir}\,(Y)\; + \; \frac 1 8\;
\eta_{\Sign}\,(Y).
\]
The latter can be calculated explicitly in terms of either integral lattice 
points or Dedekind sums; see~\cite{Nic1}. On the other hand, 
$\lambda_{\,\FO} (X)$ is equal to the equivariant Casson invariant 
$\lambda^{\tau} (Y)$, also known as the $\bar\mu$--invariant of Neumann 
and Siebenmann; see for instance~\cite{collin-saveliev} or~\cite{saveliev}. 
Our conjecture is then equivalent to showing that
\begin{equation}\label{E:barmu}
\frac 1 2\;\eta_{\Dir}\,(Y)\; +\; \frac 1 8\;\eta_{\Sign}\,(Y)\, =\, 
-\, \bar\mu (Y)
\end{equation}
for all Seifert fibered homology spheres $Y = \Sigma(a_1,\ldots,a_n)$. We 
succeeded in checking that~\eqref{E:barmu} is true by a direct calculation 
with Dedekind sums; a complete proof can be found 
in~\cite{ruberman-saveliev:mubar}.

\bigskip 


\appendix\label{A:vertical}
\section{Existence of special paths}
In this appendix, we verify that any regular path of metrics and perturbations can be homotoped rel endpoints to a path $(g_I,\beta_I)$ so that near any points $t$ for which $\M^0 (X,g_t,\beta_t)$ is non-empty, the metric is constant.  This was stated as Theorem~\ref{T:special} in Section~\ref{S:regular}.

Let us write $\met$ for the space of Riemannian metrics on $X$, with the 
$C^k$ topology for some sufficiently large $k$.  As in 
Section~\ref{S:regularity}, form the space $\tmZ_{\met} \subset \met 
\times \widetilde\B$ consisting of quadruples $(g,A,s,\phi)$ with $D^+_A(X,g)\,
(\phi) = 0$. The proof of Lemma 27.1.1 in~\cite{KM} that shows that $\tmZ$ 
is a Hilbert submanifold of $\widetilde \B$ for any metric $g$, goes through 
with little change to show that $\tmZ_{\met}$ is a Hilbert submanifold of 
$\met \times \widetilde\B$. The equation $s = 0$ defines a codimension one 
submanifold $\p \mZ_{\met}\, \subset \,\tmZ_{\met}$. The projection $(g,A,s,
\phi)\to g$ induces submersions $\tmZ_{\met} \to \met$ and $\p \mZ_{\met} 
\to \met$.

Let $\Omega_{\met}$ be the subspace of $\met \times\;\Omega^2 (X,i\R)$ 
comprised of the pairs $(g,\omega)$ such that $\omega$ is self-dual with 
respect to the metric $g$. Projection onto the first factor makes it into 
a Hilbert bundle $\pi: \Omega_{\met} \to \met$. As such, it is isomorphic 
to the parameterized perturbation space $\P_{\met}$, which is defined as 
the Hilbert bundle over $\met$ comprised of the pairs $(g,\beta) \in \met 
\times\,\Omega^1 (X,i\R)$ such that $d^*\beta = 0$ and $\beta \in 
\H^1(X,i\R)^{\perp}$ with respect to $g$. The bundle isomorphism sends 
$(g,\beta) \in \P_{\met}$ to $(g,d^{+_g}\beta) \in \Omega_{\met}$. 

The above maps can be included into the commutative diagram
\begin{equation}\label{E:cd2}
\begin{CD}
\tmZ_{\met} @ >\swz_{\met} >> \Omega_{\met} \\
@VVV @ VV\pi V \\
\met @> 1 >> \met
\end{CD}
\end{equation}
where the map $\swz_{\met}$ is given by $\swz_{\met}\,(g,A,s,\phi) = 
(g, F^{+_g}_A - s^2\,\tau_g (\phi))$. The restriction of $\swz_{\met}$ 
to $\p\mZ_{\met}$ will be denoted $\p\swz_{\met}$. Note that the 
commutative diagram~\eqref{E:cd1} is just the pull back of the 
diagram~\eqref{E:cd2} via the map $g_I: I \to \met$.

It follows from Theorem~\ref{T:reg} that a generic path $\gamma: I \to 
\P_{\met}$ connecting regular values $\gamma_0 = (g_0,\beta_0)$ and 
$\gamma_1 = (g_1,\beta_1)$ of $\swz_{\met}$ has the property that the 
composite path $I\to \P_{\met}\to \Omega_{\met}$, which we will also 
call $\gamma$, is transverse to both $\p\swz_{\met}$ and $\swz_{\met}$.
We will say that $\gamma$ is \emph{vertical} near its intersection point 
with $\im (\p\swz_{\met})$ if its metric component is constant in a 
neighborhood of that point. 

\begin{theorem}\label{T:vertical}
Any generic path $\gamma: I \to \P_{\met}$ as above can be homotoped rel
its endpoints to a generic path that is vertical near every intersection 
point with $\im(\p\swz_{\met})$.
\end{theorem}

This theorem is a re-statement of Theorem~\ref{T:special}. Its proof 
will have two steps; one local, and a second involving some global 
arguments. The first step is contained in the following three lemmas, 
all of which are concerned with a point $z \in \p \mZ_{\met}$ such that 
$(\p\swz_{\met})(z) = \gamma_0 = (g,w) \in \Omega_{\met}$, after a proper 
reparametrization of $I$. 

\begin{lemma}\label{L:injective}
Suppose that $\gamma:[-1,1] \to \P_{\met}$ is a generic path, and let 
$z \in \p\mZ_{\met}$ be such that $\swz_{\met} (z) = \gamma_0$. Then 
the differential at $z$ of the map $\p\swz_{\met}: \;\p\mZ_{\met} \to 
\,\Omega_{\met}$ is injective. 
\end{lemma}

\begin{proof}
In our situation, the map $\swz_{\met}$ has index zero, and its 
restriction $\p\swz_{\met}$ has index $-1$. Because of transversality
to $\gamma$, the differential at $z$ of the map $\p\swz_{\met}$ has 
image of codimension one, hence it must be injective.
\end{proof}

By the inverse function theorem, there is a neighborhood $U$ of $z$ in 
$\p\mZ_{\met}$ such that the restriction of $\p\swz_{\met}$ to $U$ is 
an embedding into an open ball $W \subset \Omega_{\met}$. Let us write 
$V = (\p\swz_{\met})(U)$. 

\begin{lemma}\label{L:vert-vector}
There is a unique unit vector $\nu \in T_w \,\Omega_{\met}$ such that 
\begin{enumerate}
\item $\nu$ is vertical, that is, $\nu \in \ker (D_w \pi)$ with $\pi:
\,\Omega_{\met} \to \met$, 
\item $\mathbb R\cdot \nu \,\oplus\, T_w V = T_w \,\Omega_{\met}$, and
\item $\nu = a \gamma'(0) +Y$ for some $Y \in T_w V$ and $a>0$. \label{side}
\end{enumerate}
\end{lemma}

\begin{proof}
Since  $\partial{\mathcal Z}_{\mathcal R} \to \mathcal R$ is a submersion, the 
commutativity of the diagram~\eqref{E:cd2} implies that the differential at 
$w$ of $\pi: \Omega_{\met} \to \met$ is surjective when restricted to $T_w V 
\subset T_w \,\Omega_{\met}$. Since $V$ has codimension one in $\Omega_{\met}$, 
the result follows. 
\end{proof}

In the proof of Proposition~\ref{P:local-cross} below, we will need an
estimate saying that the Seiberg--Witten map $\p\swz_{\met}$ is close to 
its derivative.  

\begin{lemma}\label{L:tangent} There is a constant $C$ such that 
\begin{equation}\label{E:tangent}
\|\,(\p\swz_{\met})(z_1) -D_z\,(\p\swz_{\met})(z_1-z)\,\|\;\leq C\;
\|\,z_1-z\,\|^2
\end{equation}
for $\| z_1-z\|$ sufficiently small. In this inequality, the distances are 
computed in $\met \times \widetilde\B$ and $\met \times\,\Omega^2(X,i\R)$,
respectively.
\end{lemma}

\begin{proof}
According to Taylor's theorem for smooth functions on Banach spaces 
(cf.~\cite[\S 4.5]{zeidler:applications}), an estimate of the 
form~\eqref{E:tangent} holds whenever the second derivative of 
$\p\swz_{\met}$ is uniformly bounded in a ball around $z$. The map 
$\p\swz_{\met}$ is linear in the connection $A$ (and independent of 
$\phi$ since $z\in \p\mZ_{\met}$), so the only issue is the dependence 
on the metric $g$. Trivialize the bundle $\Omega_{\mathcal R} \to
\mathcal R$ using the maps $b_t$ as in Remark~\ref{R:id}, and observe 
that $\p\swz_{\met}$ only depends on $g$ via the Hodge star
operator. Simple linear algebra shows that this dependence is quartic, 
from which the bound follows. 
\end{proof}

A similar estimate, $\|\gamma(t) -t \gamma'(0) \|\, \leq\, C t^2$, 
holds for the path $\gamma$. These three lemmas yield the local 
statement we need.

\begin{proposition}\label{P:local-cross}
Suppose that $\gamma:[-1,1] \to \Omega_{\met}$ is a generic path in $W$ 
whose pre-image in $U$ is a single point $z$ with $w = (\p\swz_{\met})(z)
= \gamma_0$. For any sufficiently small $\delta > 0$, there is a generic 
path $\gamma_\delta$ such that $\gamma_\delta \cap V = \gamma \cap V$ 
and $\gamma_{\delta}$ crosses $V$ vertically at $w$. In addition, the 
two paths are homotopic by a homotopy supported in $(-\delta,\delta)$.
\end{proposition}

\begin{proof}
By choosing local coordinates in the ball $W$ near $w$,  we can assume 
that $V \subset W$ is a linear ball of codimension one.  In particular, 
$V$ separates $W$ into two components, and likewise the intersection 
$V \cap\, \p W$ separates $\p W$ into two components. Since $\gamma$ is 
transversal to $V$ at $w$, it must cross from one side of $V$ to the 
other. Choosing $W$ small enough, we can assume that 
\begin{enumerate}
\item $\gamma^{-1} (W)$ is a single interval $(-\delta,\delta)$ for 
some $\delta > 0$,
\item a small positive multiple $a_+\nu$ of the vector $\nu$ from 
Lemma~\ref{L:vert-vector} lies on the same side of $V$ as $\gamma(t)$ 
for $t \in (0,\delta)$,
\item a small negative multiple $a_-\nu$ lies on the same side of $V$ 
as $\gamma(t)$ for $t \in (-\delta,0)$, \, and
\item the multiples $a_\pm \nu$ belong to $\p W$.
\end{enumerate}

\noindent
Now we can construct a path $\gamma_\delta$  as follows. It is equal to 
$\gamma$ outside of $(-\delta,\delta)$. From $-\delta$ to $-\delta/2$, 
it follows a path in $\p W$ (avoiding $V \cap \p W$) from $\gamma(-\delta)$ 
to $a_- \nu$. From $-\delta/2$ to $\delta/2$, it follows the vertical line 
from $a_- \nu$ to $a_+\nu$. Finally, from $\delta/2$ to $\delta$, it 
follows a path in $\p W$ (avoiding $V \cap \p W$) from $a_+\nu$ to $\gamma
(\delta)$. See Figure~\ref{fig:local-int}.
\end{proof}

\begin{figure}[!ht] 
\centering
\psfrag{om2}{$\Omega^2_{+}$}
\psfrag{gd}{$\gamma_{\delta}$}
\psfrag{apn}{$a_+\nu$}
\psfrag{amn}{$a_-\nu$}
\psfrag{gam}{$\gamma$}
\psfrag{gam'}{$\gamma'(0)$}
\psfrag{nu}{$\nu$}
\psfrag{W}{$W$}
\psfrag{w}{$w$}
\psfrag{met}{$\met$}
\psfrag{swz}{$V$}
\psfrag{tsw}{$T_w V$}
\psfrag{P}{$\P$}
\includegraphics[scale=.75]{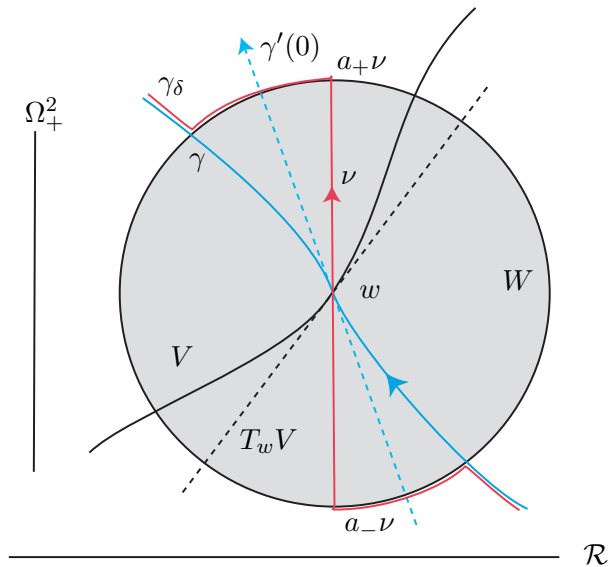} 
\caption{Local modification of $\gamma$}
\label{fig:local-int} 
\end{figure} 

The main issue in deriving Theorem~\ref{T:vertical} from the above local 
result, which is Proposition~\ref{P:local-cross}, is that a given 
intersection point $w \in \im (\p\swz_{\met}) \cap \gamma$ may have 
several pre-images in $\p\mZ_{\met}$. There are at most finitely many such 
pre-images, which follows from the map $\swz_{\met}$ being proper, as 
proved in~\cite[Theorem 5.2.1]{KM}. 

Denote the points in $(\p\swz_{\met})^{-1}(w)$ by $z_k \in \p\mZ_{\met}$, 
$k = 1,\ldots, n$. Choose an open ball $W$ centered at $w$ with the 
property that, for each $k$, the point $z_k$ has a neighborhood $U_k 
\subset \p\mZ_{\met}$ such that the restriction of $\p\swz_{\met}$ to 
$U_k$ is an embedding into $W$. Note that the sheets $V_k = 
(\p\swz_{\met})(U_k)$ meeting at $w$ are all embedded but they might be 
transverse or tangent to each other. 

\begin{lemma}
If the ball $W$ as above is chosen sufficiently small then 
\[
(\p\swz_{\met})^{-1}(W)\; \subset\; \bigsqcup_k\; U_k.
\]
\end{lemma}

\begin{proof}
Suppose to the contrary that there is no such $W$. Then there is a sequence 
of points $u_i \in \p\mZ_{\met}$ away from the $U_k$ with bounded energy 
(again, referring to~\cite[Theorem 5.2.1]{KM}) and with $(\p\swz_{\met})(u_i)$ 
converging to $w$. It follows that a subsequence of the $u_i$ converges in 
$\p\mZ_{\met}$, and then of course it must converge to one of the $z_k$. 
This contradiction establishes the claim.
\end{proof}

Next, we claim that we may assume that all of the $V_k$ are tangent at their 
common intersection point $w$. Suppose that $T_w V_i \neq T_w V_j$ for some 
$i \neq j$. Since both of these subspaces have codimension one (cf. the 
proof of Lemma~\ref{L:injective}), they must in fact be transverse. It 
follows that the intersection $V_i \cap V_j$ must be a codimension $2$ 
submanifold near $w$. A small perturbation of $\gamma$ will then suffice to 
avoid all such double point sets.

Finally, we are in the situation where at any intersection point $w$, all of 
the finitely many $V_k$ that meet $W$ are tangent. It follows that the vector
$\nu$ constructed in Lemma~\ref{L:vert-vector} lies on the same side of all 
of the $V_k$, and so (if $W$ is sufficiently small) the curve constructed by 
the local deformation in Proposition~\ref{P:local-cross} meets all sheets 
vertically.  Since the deformation takes place in $W$, no new intersections 
have been introduced, and Theorem~\ref{T:vertical} is proved.


\end{document}